\documentclass[12pt]{amsart}
\usepackage{color}

\usepackage{geometry}                
\usepackage{graphicx}
\usepackage{amssymb}
\usepackage{amsmath}
\usepackage{epstopdf}

\usepackage{amsmath}
\usepackage{amssymb}
\usepackage{epsfig}
\usepackage{hyperref}
\newtheorem{Proposition}{Proposition}
  \newtheorem{Remark}{Remark}
  
  \newtheorem{Lemma}[Proposition]{Lemma}
  
  \newtheorem{Theorem}{Theorem}
 
 \newtheorem{Definition}[Proposition]{Definition}

\def\blackslug{\hbox{\hskip 1pt \vrule width 4pt height 8pt depth 1.5pt
\hskip 1pt}}
\def\qed{\quad\blackslug\lower 8.5pt\null\par}

\def\Re{\mathrm{Re}}
\def\Im{\mathrm{Im}}

\makeindex

\title[Hybrid basis scheme for close-to-touching discs]{Hybrid basis scheme for computing electrostatic fields exterior to close-to-touching discs}

\author{D. G. Crowdy, S. Tanveer$^{(*)}$ \& T. DeLillo$^{(**)}$}

\date{}

\begin{document}
\maketitle

\begin{center}
Department of Mathematics \\
Imperial College London \\
180 Queen's Gate \\
London, SW7 2AZ, U.K.

\vskip 0.1truein 
{\tt d.crowdy@imperial.ac.uk} \\

\vskip 0.2truein
$^*$Department of Mathematics \\
Ohio State University, \\
Columbus, OH 43210, USA

\vskip 0.1truein 
{\tt tanveer@math.ohio-state.edu} \\

\vskip 0.2truein
$^{**}$Department of Mathematics \\
Wichita State University, \\
Wichita, KS 67260, USA

\vskip 0.1truein 
{\tt delillo@math.wichita.edu} \\
\end{center}

\vskip 0.5truein

\begin{abstract}
This paper presents a simple and effective new
numerical scheme for the computation of
electrostatic fields exterior to a collection
of close-to-touching discs.
The method is presented in detail for the 
two-cylinder case. The key idea is to represent the required complex
potential using a hybrid
set of basis functions comprising the usual Fourier-Laurent
expansion about each circle centre complemented by a subsidiary
expansion in a variable associated with conformal mapping
of the physical domain to a concentric annulus domain.
We also rigorously prove that there is a representation of the solution in the hybrid basis with faster decay rate of coefficients than is obtained by using a non-hybrid basis,
thereby providing a rationalization for the success of the method.
The numerical scheme is easy to implement and adaptable
to the case of multiple close-to-touching cylinders.

\end{abstract}

\vfill\eject

\section{Introduction}

This paper is concerned with a classical
problem of great importance
in the manufacture of composite materials (see \cite{Milton} and references therein), as well as in many
other applications. It is the determination of the electrical
transport
properties in two-dimensional media containing a set of
embedded inclusions which may be close-to-touching.
It is assumed that there is a uniform background field.
Lord Rayleigh \cite{Rayleigh} was one of the first to consider
such issues in the context of a regular periodic array of circular
inclusions.
In this paper we focus on the two-dimensional case, and in ways of calculating
the field exterior to the inclusions.

This problem becomes very singular as the inclusions become
close-to-touching and this circumstance presents formidable computational
problems that have been the subject of much research.
In cases there the geometry is sufficiently simple,
asymptotic methods can be used to get quantitative
insight into this case \cite{McP1} \cite{McP2} \cite{McP3}.
The most successful numerical schemes in two-dimensions
are
integral equation methods. These, nevertheless, also encounter difficulties
when the inclusions (or discs) are close-to-touching.
Greengard \& Moura \cite{Green1} employed fast
multipole-accelerated integral equation methods whose
cost grows linearly in the number of unknowns. 
Since their method does not make any special
provision for the singular
nature of the problem as the discs get close together, these
methods quickly become expensive.
Helsing \cite{Hel1} \cite{Hel2} has developed 
fast multipole-accelerated iterative schemes for inclusions
of arbitrary shape.
More recently, focussing on the case of arrays of close-to-touching
discs, Cheng \& Greengard \cite{CG1} (see also \cite{CG2}) have presented a  method which
marries ideas from the ``method of images'' to various
integral equation techniques. Their scheme relies on including
multipole expansions, of various orders, about successive generations
of reflections (or ``images'') of the centres of the discs.
These
additional reflected terms allow even very singular cases
to be resolved to essentially arbitrary accuracy. While the number
of multipole expansions involved can become very large, the actual
number of unknowns associated with each inclusion is kept small
by making use of known reflection operators which 
allow the coefficients of the expansions about the
next-generation
images to be inferred analytically from the parent coefficients.

The aim of this paper is to present an apparently new numerical
scheme to address the challenge of close-to-touching inclusions.
A key advantage of the method is that it is conceptually
simple and very
easy
to implement. The central idea is to make use of a strategically chosen over-complete
basis set which we refer to as a ``hybrid basis''.
The two cylinder problem is studied in detail.
In \S \ref{Saleh} 
we prove, for the two cylinder problem, that there exists a representation of the solution in
terms of a hybrid basis having faster decay rate properties than that obtained 
through a more traditional non-hybrid method. This provides
a rationalization of the empirical observation of superior performance
when the so-called ``hybrid basis scheme'', to be explained in the next two sections, is used.

%

\section{Problem formulation}

Suppose we wish to solve the problem for the electrostatic
field exterior to $M$ circular discs,
in an $(x,y)$-plane,
with the field strength in the far-field
tends to the value $E$ at an angle
$\chi$ to the $x$-axis.
We will repose the problem using a complex variable formulation.
The problem is mathematically equivalent to finding
the function $w(z)$, analytic and single-valued outside a collection of circular
discs, with
\begin{equation}
w(z) \sim E_0 z + {\mathcal O}(1),~~~{\rm as}~~z \to \infty,
\end{equation}
where $E_0 = E e^{-{\rm i} \chi} \in \mathbb{C}$ and 
\begin{equation}
{\rm Re}[w(z)] = \biggl \lbrace \begin{array}{ll}
0, & {\rm on~} C_1, \\
\gamma_j, & {\rm on}~C_j,~~j=2,3,...,M.
\end{array}
\label{eq:eqBC}
\end{equation}

An identical problem arises in ideal
fluid dynamics \cite{Burnside} \cite{Exact}; it is the problem 
of finding the uniform flow past a collection of discs
where the far-field flow has speed $U$ and angle $\chi$ to the
positive real axis. This reduces to finding the function $w(z)$,
analytic and single-valued outside a collection of circular
discs, 
with
\begin{equation}
w(z) \sim U_0 z + {\mathcal O}(1),~~~{\rm as}~~z \to \infty,
\end{equation}
where $U_0 = U e^{-{\rm i} \chi} \in \mathbb{C}$ and
\begin{equation}
{\rm Im}[w(z)] = \biggl \lbrace \begin{array}{ll}
0, & {\rm on~} C_1, \\
\gamma_j, & {\rm on}~C_j,~~j=2,3,...,M.
\end{array}
\label{eq:eqBC}
\end{equation}
In either problem,
the set $\lbrace \gamma_j | j=2,...,M \rbrace$ is determined as part
of the solution.
In what follows, we consider the two physical problems just stated as
essentially the same; indeed, they are the same once the identification
$w \mapsto {\rm i} w$ is made.
We will solve the problem of uniform flow past 
the discs since convenient analytical solutions to this problem 
are known (see, for example, Crowdy \cite{Exact}) and these will
be used for bench-marking purposes.

\section{Three numerical methods}

To illustrate the ideas behind the new method, it is expedient to consider
three different numerical schemes for the computation of $w(z)$. These
will be referred to as the $z$-scheme, the $\zeta$-scheme and the hybrid
scheme. It is the third hybrid scheme which constitutes the new contribution
of this paper.

\subsection{Fourier-Laurent method (the ``$z$ scheme'')}

The first scheme, which will be referred to as
the ``$z$ scheme'', is 
a standard scheme 
commonly used in the numerical computation of fields in multiply
connected domains. 
Prosnak \cite{Prosnak}, for example,
makes liberal use of it in the computation of
fluid flows in multiply connected geometries.

We describe this method in the context of a two cylinder example.
Consider two circular discs, each of radius $s$, centred at $\pm d$ where $d > s$.
There is no loss in generality in setting $d=1$ and $|U_0|=1$ since all
length and time scales can be non-dimensionalized with respect to $d$ and
$d/|U_0|$ respectively.

A natural approach is to write
the following Fourier-Laurent expansion about the centres of the discs:
\begin{equation}
w(z) = U_0 z + a_0 + \sum_{k=1}^\infty {a_k s^k \over (z-d)^k} +
 \sum_{k=1}^\infty {b_k s^k \over (z+d)^k},
 \label{eq:eqM1}
\end{equation}
where the coefficients $a_0, \lbrace a_k, b_k | k =1,2,... \rbrace$ are
to be determined and the first term clearly enforces the far-field condition.
The representation (\ref{eq:eqM1}) can be substituted into the
boundary conditions (\ref{eq:eqBC}) and evaluated at a set of collocation
points
which are most naturally chosen to be a set of
equally spaced points around the boundary
circles.
The number of collocation points must be taken to be at least as large
as the number of unknown coefficients retained in the truncation. 
Since the boundary conditions are linear in
the unknown coefficients, the latter can be found by a straightforward
least-squares
solution of the over-determined system. 

%
%

\subsection{Conformal mapping method (the ``$\zeta$ scheme'')}

A second numerical scheme makes use of a conformal mapping.
Consider 
the conformal mapping from an annulus $\rho < |\zeta| < 1$ to the unbounded
region exterior to the two discs in the $z$-plane. It is a simple M\"obius map given by
\begin{equation}
z (\zeta) = A \left ( {\zeta-\sqrt{\rho} \over \zeta+\sqrt{\rho}} \right ),
\label{eq:eq1}
\end{equation}
where
\begin{equation}
\rho = {1- (1-(s/d)^2)^{1/2} \over 1 + (1-(s/d)^2)^{1/2}}
\label{eq:eq2}
\end{equation}
and 
\begin{equation}
A = d \left ({1-\rho \over 1+\rho} \right )= \sqrt{d^2-s^2}.
\label{eq:eq3}
\end{equation}
(There is an abuse of notation here -- and throughout -- in that $z$ is used both
as a coordinate point and as a conformal mapping function.)
The point $\zeta=-\sqrt{\rho}$ has been chosen to map to infinity in the $z$-plane.
The inverse function to (\ref{eq:eq1}) is 
\begin{equation}
\zeta(z) = \sqrt{\rho} \left ( {A+z \over A-z} \right ).
\label{eq:eqINV}
\end{equation}
In this second method, referred to as the ``$\zeta$ scheme'', 
we use the modified representation of
$w(z)$ given by
\begin{equation}
w(z) = U_0 z + C 
 + \sum_{k=1}^\infty c_k [\zeta(z)]^k
  + \sum_{k=1}^\infty {d_k \rho^k \over [\zeta(z)]^{k}},
 \label{eq:eqM2a}
\end{equation}
where the coefficients
$C, \lbrace a_k, b_k, c_k, d_k | k =1,2,... \rbrace$ are now
to be determined. 
The terms
\begin{equation}
\sum_{k=1}^\infty c_k [\zeta(z)]^k
  + \sum_{k=1}^\infty {d_k \rho^k \over [\zeta(z)]^{k}}
  \label{eq:eqM3}
\end{equation}
can be recognized as a Laurent
series capable of representing any function that is analytic and single-valued
in the annulus $\rho < |\zeta|< 1$ and hence, under the conformal mapping,
in the region exterior to
the two discs. As before, the easiest strategy to find the unknown coefficients is to solve
an over-determined system by a least-squares method.

\subsection{New method (the ``hybrid basis scheme'')}

The new numerical method will be called the ``hybrid basis scheme''.
The idea is to make use of the modified representation of
$w(z)$ given by
\begin{equation}
w(z) = U_0 z + C + \sum_{k=1}^\infty {a_k s^k \over (z-d)^k} +
 \sum_{k=1}^\infty {b_k s^k \over (z+d)^k}
 + \sum_{k=1}^\infty c_k [\zeta(z)]^k
  + \sum_{k=1}^\infty {d_k \rho^k \over [\zeta(z)]^{k}},
 \label{eq:eqM2}
\end{equation}
where the coefficients
$C, \lbrace a_k, b_k, c_k, d_k | k =1,2,... \rbrace$ are now
to be determined. 
(\ref{eq:eqM2}) 
provides a
representation of the required function $w(z)$ that
is uniformly valid everywhere in this annulus and, hence, everywhere in
the flow region.

{\em In theory}, the hybrid basis used in
(\ref{eq:eqM2}) is overcomplete;
{\em in practice}, however, since all the infinite sums must be truncated
at some level, it is possible that the hybrid representation
(\ref{eq:eqM3}) may offer numerical advantages
over the
$\zeta$-scheme or the $z$-scheme used separately.
We will now show that this is indeed true and, moreover, that
the concomitant advantages are dramatic.

We have found that the method of determination of the unknown coefficients in
the hybrid representation requires care, otherwise the potential advantage is easily
lost.
To determine the unknown coefficients
in (\ref{eq:eqM2}),
we truncate, at order $N$, each of the four infinite sums in (\ref{eq:eqM2})
and solve, by a least squares procedure, 
an over-determined linear system. This system is found 
by substituting 
(\ref{eq:eqM2}) into the boundary conditions
(\ref{eq:eqBC}) and evaluating them at ${\mathcal N}$ collocation points
on the boundary of each of {\em four} discs: ${\mathcal N}$ points are taken
to be equi-spaced on each of the two circles 
$|z \pm d| =s$ together with ${\mathcal N}$ 
points equi-spaced around each of the two circles $|\zeta|=1, \rho$. 
In order that the system is overdetermined we clearly need the number ${\mathcal N}$
of collocation points on each of the four circles to
exceed $2N$ since the boundary conditions on each circle are real and $N$ complex
coefficients associated with each circle are to be determined (as well as the single real
constant
$\gamma_2$);
typically we take ${\mathcal N} = 4N$, 
but the method is found to be insensitive to this choice.
The chosen collocation points are illustrated on the left in Figure \ref{Fig1aa};
to the right the images, under the map (\ref{eq:eq1}) and its inverse (\ref{eq:eqINV}),
of the collocation points
in the ``other'' plane are shown.
On the physical disc boundaries, the equi-spaced collocation
points on the $\zeta$ circles
naturally crowd around the area where they are most
needed -- in the region where the discs are closest together.
On the other hand, the image of uniformly spaced out points on
the $z$-circles crowd around the negative real axis 
in the $\zeta$-plane (recall that $\zeta=-\sqrt{\rho}$ is the preimage of the point at
infinity in the physical plane). Intuitively, then, our choice of hybrid basis scheme
has the boundaries of the discs ``well covered''.

The above choice
of collocation points was found to be crucial to the success of the method:
the hybrid representation
of the solution appears only to be effective provided the collocation points
are chosen as described above.

\begin{figure}
\begin{center}
\includegraphics[scale=0.6]{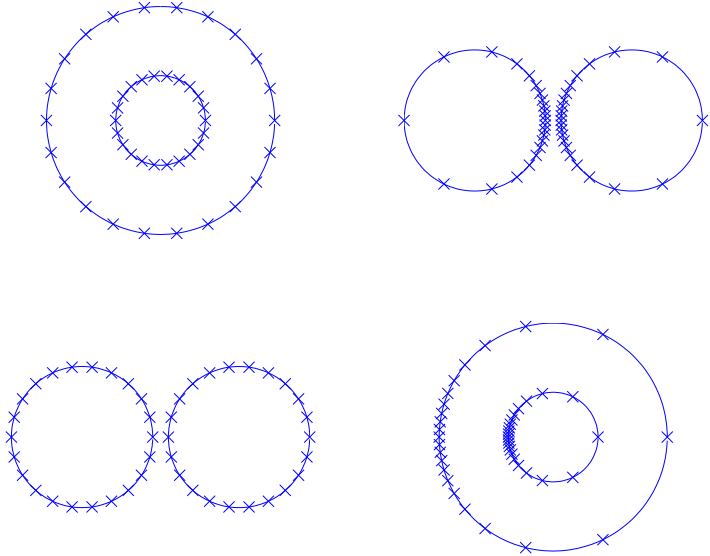}
\caption{The images (shown right) under the forward and inverse
conformal mappings of equi-spaced pre-image points on the circles
on the $\zeta$ and the $z$ planes (shown left). The latter
are used as collocation points for the hybrid basis scheme.
\label{Fig1aa}}
\end{center}
\end{figure}

\section{Exact solution}

There is an exact solution, described in \cite{Exact},
 for the two-disc problem which can be
used to test
the accuracy of the numerical scheme; it also
forms the basis of our analysis in \S \ref{Saleh}.
This solution is given as follows.
Define the two functions
\begin{equation}
P(\zeta) \equiv (1-\zeta) \prod_{k=1}^\infty (1-\rho^{2k} \zeta)(1-\rho^{2k} \zeta^{-1}),\qquad
K(\zeta) = {\zeta P_\zeta(\zeta) \over P(\zeta)}.
\label{eq:eq4}
\end{equation}
Then the exact solution for $W(\zeta) \equiv w(z(\zeta))$, satisfying the (arbitrarily
chosen) normalization
that $W(-1)=0$, is
\begin{equation}
W(\zeta) = - 2 A U_0 \left ( K(\sqrt{\rho}^{-1}) - K(-\zeta \sqrt{\rho}^{-1}) \right )
+ 2 A {\overline U_0} \left ( K(\sqrt{\rho}) - K(-\zeta \sqrt{\rho}) \right ).
\label{eq:eqEX}
\end{equation}
This
exact solution can be computed to arbitrary accuracy by truncating
the infinite product (\ref{eq:eq4}) at the appropriate level.
It can be shown, directly from its definition (\ref{eq:eq4}), that $P(\zeta,\rho)$
satisfies the functional relations
\begin{equation}
P(\rho^2 \zeta, \rho) = P(1/\zeta,\rho) = - \zeta^{-1} P(\zeta,\rho).
\end{equation}
From these it is then easily deduced that
\begin{equation}
K(\rho^2 \zeta,\rho) = K(\zeta,\rho) - 1, \qquad K(1/\zeta,\rho) = 1- K(\zeta,\rho).
\label{eq:eqKProp}
\end{equation}
The identities (\ref{eq:eqKProp}) 
can be used to directly verify that (\ref{eq:eqEX}) is the required solution.

\begin{figure}
\begin{center}
\includegraphics[scale=0.4]{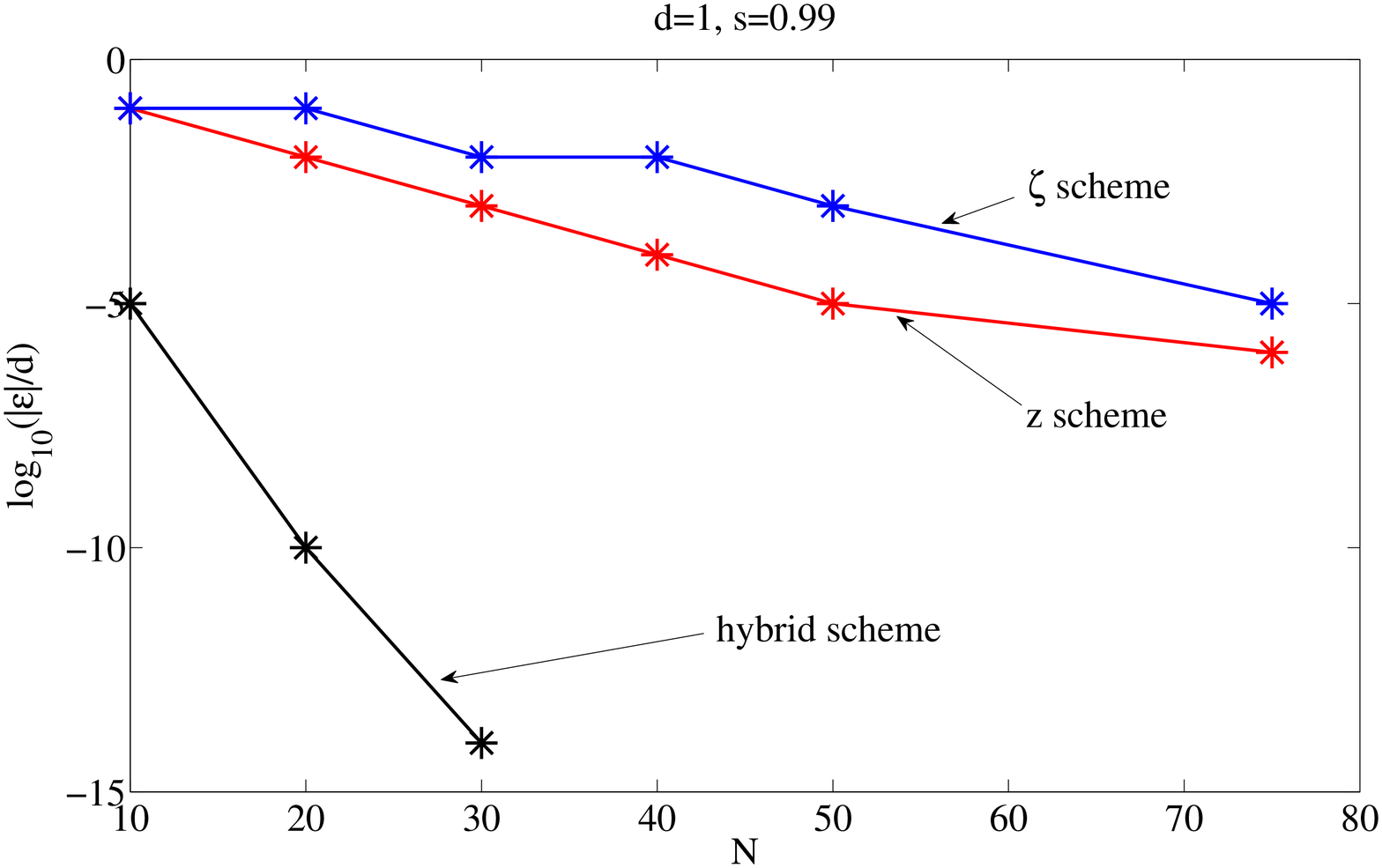}
\includegraphics[scale=0.4]{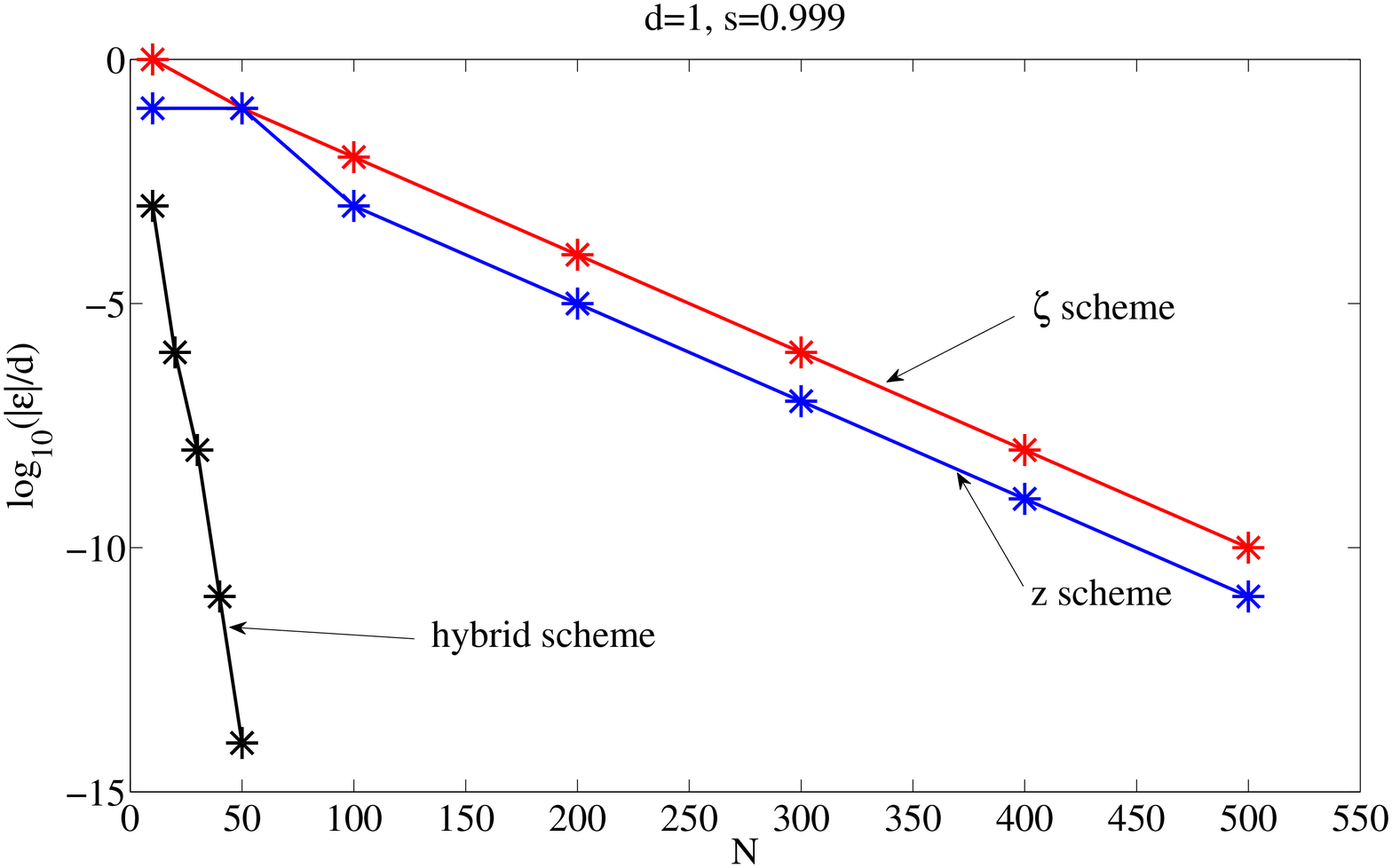}
\caption{Comparison of the performance of the three schemes in the two-disc
problem with $d=1, s=0.99$ (upper) and $d=1, s=0.999$ (lower).
The logarithm of the maximum absolute error is plotted against the number of
modes $N$ in the truncation.
\label{Fig11}}
\end{center}
\end{figure}


\section{Performance}
To compare the performance of the three methods, the upper graph in
Figure \ref{Fig11}
shows the logarithm of the maximum absolute error in the
numerical solutions compared with the exact solution plotted as a function
of the truncation level $N$ for the case $d=1, s=0.99$ and $U_0=e^{{\rm i}\pi/4}$ so that 
the separation of the two discs is $0.02$.
The errors were computed by calculating the difference between
$w(z)$ as given by the exact solution (\ref{eq:eqEX}) (evaluated on the
boundaries of the two discs) to the values given by the numerical schemes.
It is clear that, while both the $z$-scheme and $\zeta$-scheme give comparable
accuracy at the same level of truncation, the hybrid scheme offers
dramatic increases in accuracy at much smaller levels of truncation.
This feature becomes even more pronounced at smaller
separation distances: Figure \ref{Fig11} also
shows results for
$d=1, s=0.999$ and $U_0=e^{i\pi/4}$ so that 
the separation of the two discs is $0.002$, an order of magnitude smaller.

The advantages of the hybrid basis scheme
are seen even more clearly in Figure \ref{Fig10}.
At each value of the disc separation (horizontal axis)
the vertical axis shows the number of modes required to attain
accuracy of $10^{-6}$ (as compared with the exact solution).
While around $1000$ modes are needed for both the
$z$-scheme and the $\zeta$-schemes in order 
to attain the required accuracy when the disc separation is of the order
of $10^{-4}$, the hybrid scheme attains the same accuracy with only $140$ modes
even when the separation is as small as $10^{-6}$.
For such small separations, both the
$z$-scheme and the $\zeta$-schemes
become unfeasible (requiring extremely large numbers of modes,
which is why the results for these methods have been omitted
from Figure \ref{Fig10}).

\begin{figure}
\begin{center}
\includegraphics[scale=0.4]{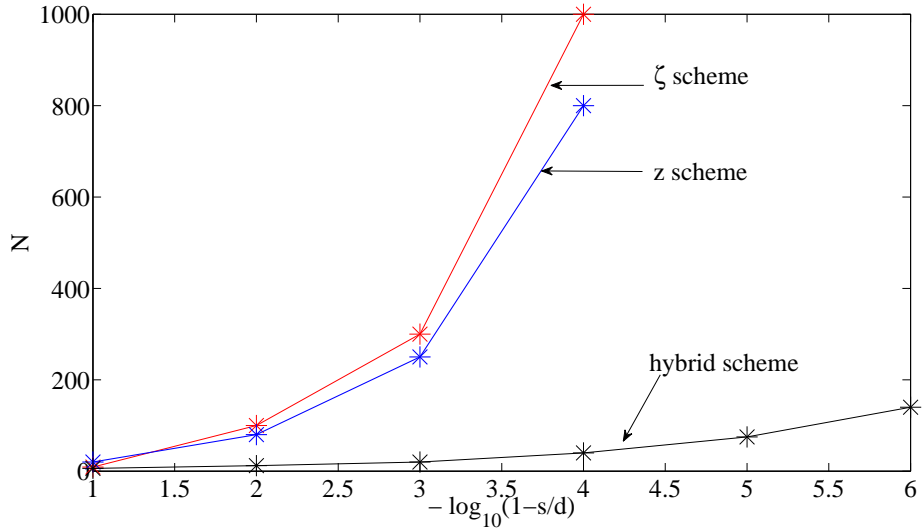}
\caption{Comparison of the performance of the three schemes in the two-disc
problem. 
At each value of the disc separation (horizontal axis), 
the vertical axis shows the number of modes required to attain
accuracy of $10^{-6}$ (compared to the exact solution).
\label{Fig10}}
\end{center}
\end{figure}

\section{Two-scale analysis \label{Saleh}}

The aim of this section is to
demonstrate why a hybrid basis scheme involving series
expansions in both the $\zeta$
and $z$ variables might be expected to
perform better than separate expansions in either variable, as the
foregoing numerical evidence has shown.
We show the hybrid scheme
has better decay properties of its 
coefficients than a power series in either the $\zeta$
or $z$ variables separately. We will prove the following theorem:
\begin{Theorem}
\label{Thm1}
There exists a representation of complex potential $w$ in the following form: 
\begin{equation}
w = U_0 z + a_0+ \sum_{j=1}^\infty \left ( a_j \zeta^j + b_j \rho^{j} \zeta^{-j} \right )  
+ \sum_{j=1}^\infty \left [ {c_j s^j \over (z-d)^j} + {d_j {s^j} \over (z+d)^j} \right ],
\end{equation}
where for any $k \ge 1$,
\begin{equation}
j^k |a_j | \ , j^k |b_j| \ , j^k |c_j| \ , j^k |d_j | \le \frac{M_k}{T^{k/2+1}} \ , 
\end{equation} 
where $T= (1/\pi) \log \rho^{-1}$ scales as {$\sqrt{1-s/d}$}
and $M_k$ is independent of $T$.
However, when either of the sets $\left \{ a_j, b_j \right \} $ or $\left \{ c_j, d_j \right \}$ is chosen to be zero,
then the best possible bounds in the above scale as $1/T^{k+1}$. 

Further, 
there exists a logarithmic decomposition in the form
\begin{equation}
w(z) = U_0 z + a_0+ \sum_{j=1}^\infty \frac{c_j s^j}{(z-d)^j} + \frac{d_j s^j}{(z+d)^j} + 
\frac{A}{\pi T} \left (U_0 - {\bar U}_0 \right ) \log \left [ \frac{(z-d) (z+A)}{(z+d)(z-A)} 
\right ]
\end{equation}   
with uniform rapid decay properties
\begin{equation}
j^k |c_j| \ , j^k |d_j | \le \frac{M_k}{T}  
\end{equation} 
for any $k \ge 1$.
\end{Theorem}

The proof will make use of the exact solution (\ref{eq:eqEX}) 
known for the two-cylinder
case
and the a few preliminary Lemmas and Propositions. The proof is completed in \S \ref{proofs}.

\subsection{The function $K(\zeta,\rho)$} 

The exact solution (\ref{eq:eqEX}) for the complex velocity potential 
is given in (\ref{eq:eqEX}) in terms of the function $K(\zeta,\rho)$ 
which can be shown from its definition (\ref{eq:eq4}) to admit the infinite sum representation
%
\begin{equation}
\label{1.10}
K(\zeta) = -\frac{\zeta}{1-\zeta} 
+\sum_{k=1}^\infty \left \{ 
-\frac{\rho^{2k} \zeta}{1-\rho^{2k} \zeta}  
+\frac{\rho^{2k} \zeta^{-1}}{1-\rho^{2k} \zeta^{-1}} \right \}.
\end{equation}
$K$ is single valued function of
$\zeta$, implying that
\begin{equation}
\label{1.11}
K (\zeta e^{2 {\rm i} \pi} ) = K (\zeta),
\end{equation}
while it also satisfies the ``quasi-periodic'' property $K(\rho^2 \zeta,\rho) = K(\zeta,\rho)-1$
already given in (\ref{eq:eqKProp}).
From these properties it should be clear that there
is a connection between $K (e^{\pi \xi} )$ 
and the quasi-periodic Weierstrass zeta function in the $\xi$ variable where
\begin{equation}
\xi = \frac{1}{\pi} \log \zeta.
\end{equation}
Moreover, its derivative is associated with
the
Weierstrass $\wp$ function  with periods $ 2 T:= (2/\pi) \log \rho^{-1}$ and $2 {\rm i}$.
Indeed on use of 
certain representations of the Weierstrass $\wp$ function\footnote{We 
could also relate $K$ directly to the Weierstrass zeta function without the need for integration;
however, since certain constants have to be determined by evaluating them at half-periods in any case, there is
no particular advantage in doing this.},
it is possible to show (see appendix) that
\begin{equation}
\label{2.18.0}
K \left ( e^{\pi \xi} \right ) = \frac{\xi}{2T} + \left ( \frac{1}{2} - \frac{{\rm i}}{2 T} \right ) 
+ \frac{{\rm i}}{T} \left ( \frac{\chi}{\chi-1} \right )   
+\frac{{\rm i}}{T} \sum_{m=1}^\infty \left ( 
\frac{\mu^m \chi^{-1} }{1-\mu^m \chi^{-1} } 
-\frac{\mu^m \chi }{1-\mu^m \chi} 
\right ),
\end{equation}
where
\begin{equation}
\label{2.18.1}
T = \frac{1}{\pi} \log \rho^{-1}, \qquad \mu = e^{-2 \pi/T}, \qquad \chi = e^{{\rm i} \pi \xi/T}.
\end{equation}
It follows that on the boundary $|\zeta|=1$, which corresponds to $|z-d|=s$, with parametrizations:
\begin{equation}
\label{2.18.2}
\zeta = - e^{{\rm i} \nu}, \qquad z = d + s e^{-{\rm i} \theta},  \qquad
{\rm where} ~~\nu, \theta  \in [-\pi, \pi] \ ,
\end{equation}
we obtain
\begin{equation}
\begin{split}
K \left ( -\rho^{-1/2} \zeta \right ) &= 
\frac{{\rm i} \nu}{2 \pi T} 
+\left ( \frac{3}{4} - \frac{{\rm i}}{2 T} \right ) 
+ \frac{1}{T} \left ( \frac{e^{-\nu/T}}{1-{\rm i} e^{-\nu/T}} \right )   
\\
&+\frac{1}{T} \sum_{m=1}^\infty \left ( 
\frac{\mu^m e^{\nu/T} }{1+ {\rm i} \mu^m e^{\nu/T} } 
+\frac{\mu^m e^{-\nu/T} }{1- {\rm i} \mu^m e^{-\nu/T} }  
\right ), \\
K \left ( -\rho^{1/2} \zeta \right ) &= 
\frac{{\rm i} \nu}{2 \pi T} 
+\left ( \frac{1}{4} - \frac{{\rm i}}{2 T} \right ) 
- \frac{1}{T} \left ( \frac{e^{-\nu/T}}{1+{\rm i} e^{-\nu/T}} \right )   
\\
&-\frac{1}{T} \sum_{m=1}^\infty \left ( 
\frac{\mu^m e^{\nu/T} }{1- {\rm i} \mu^m e^{\nu/T} } 
+\frac{\mu^m e^{-\nu/T} }{1+ {\rm i} \mu^m e^{-\nu/T} }  
\right ).
\end{split}
\label{3.2a}
\end{equation}
The relationship between angles $\nu$ and $\theta$ in the $z$ and $\zeta$
domains are given by
\begin{equation}
\label{3.3}
\nu =  2 \arctan \left [ \left ( \frac{1-\sqrt{\rho}}{1+\sqrt{\rho}} \right ) \tan \frac{\theta}{2}   
\right ], \qquad 
\theta = 2 \arctan \left \{ \left ( \frac{1+\sqrt{\rho}}{1-\sqrt{\rho}} \right ) \tan \frac{\nu}{2} \right \}.
\end{equation}
Similarly, on $|\zeta|=\rho$, corresponding to $|z+d|=s$, with parameterization
\begin{equation}
\label{2.18.2a}
\zeta = - \rho e^{{\rm i} \nu}, \qquad z = -d - s e^{{\rm i} \theta},
\qquad {\rm where} ~\nu, \theta  \in [-\pi, \pi],
\end{equation}
we obtain the same relationship (\ref{3.3}) between angles $\nu$ and $\theta$ as before, 
while
\begin{equation}
\begin{split}
K \left ( -\rho^{-1/2} \zeta \right ) &= 
\frac{{\rm i} \nu}{2 \pi T} 
+\left ( \frac{1}{4} - \frac{{\rm i}}{2 T} \right ) 
- \frac{1}{T} \left ( \frac{e^{-\nu/T}}{1+{\rm i} e^{-\nu/T}} \right )   
\\
&-\frac{1}{T} \sum_{m=1}^\infty \left ( 
\frac{\mu^m e^{\nu/T} }{1- {\rm i} \mu^m e^{\nu/T} } 
+\frac{\mu^m e^{-\nu/T} }{1+ {\rm i} \mu^m e^{-\nu/T} }  
\right ),
\\
K \left ( -\rho^{1/2} \zeta \right ) &= 
\frac{{\rm i} \nu}{2 \pi T} 
+\left ( -\frac{1}{4} - \frac{{\rm i}}{2 T} \right ) 
+ \frac{1}{T} \left ( \frac{e^{-\nu/T}}{1-{\rm i} e^{-\nu/T}} \right )   
\\
&+\frac{1}{T} \sum_{m=1}^\infty \left ( 
\frac{\mu^m e^{\nu/T} }{1+ {\rm i} \mu^m e^{\nu/T} } 
+\frac{\mu^m e^{-\nu/T} }{1- {\rm i} \mu^m e^{-\nu/T} }  
\right ).
\end{split}
\label{3.4a}
\end{equation}
We note that in particular on each of the circles $|\zeta|=1$ and $|\zeta|=\rho$, we have
${\rm i} \zeta \partial_\zeta K \left ( -\rho^{-1/2} \zeta \right ) = 
\partial_\nu K \left ( -\rho^{-1/2} \zeta \right ) $ involves exponential terms in $\nu/T$ that tends
to the same constant -- namely, ${{\rm i}}/[{2 \pi T}]$ --
exponentially outside an
$O(T)$ neighborhood of $\nu=0$. Because {of the} $\nu=O(T)$ scale,  it is clear also that this a series
representation in the form $W-U_0 z=\sum_{k} \left ( a_j \zeta^j + b_j \rho^j \zeta^{-j} \right ) $ will have
poor decay properties; indeed since bounds on each of
$\partial_{\nu}^{k} \left (W (-e^{i \nu} )-U_0 Z (-e^{i \nu} ) \right )$
$\partial_{\nu}^{k} \left (W (-\rho e^{i \nu} )-U_0 Z (-\rho e^{i \nu} ) \right )$
scales as $T^{-k-1}$, from well-known properties of Fourier coefficients,
it follows we will obtain poor estimates 
$j^k |a_j| \ , j^k |b_j| \le M_k/T^{k+1}$. We now seek to alleviate this through use of
a hybrid representation.

\begin{figure}
\begin{center}
\hskip -0.4truein\includegraphics[scale=1.05]{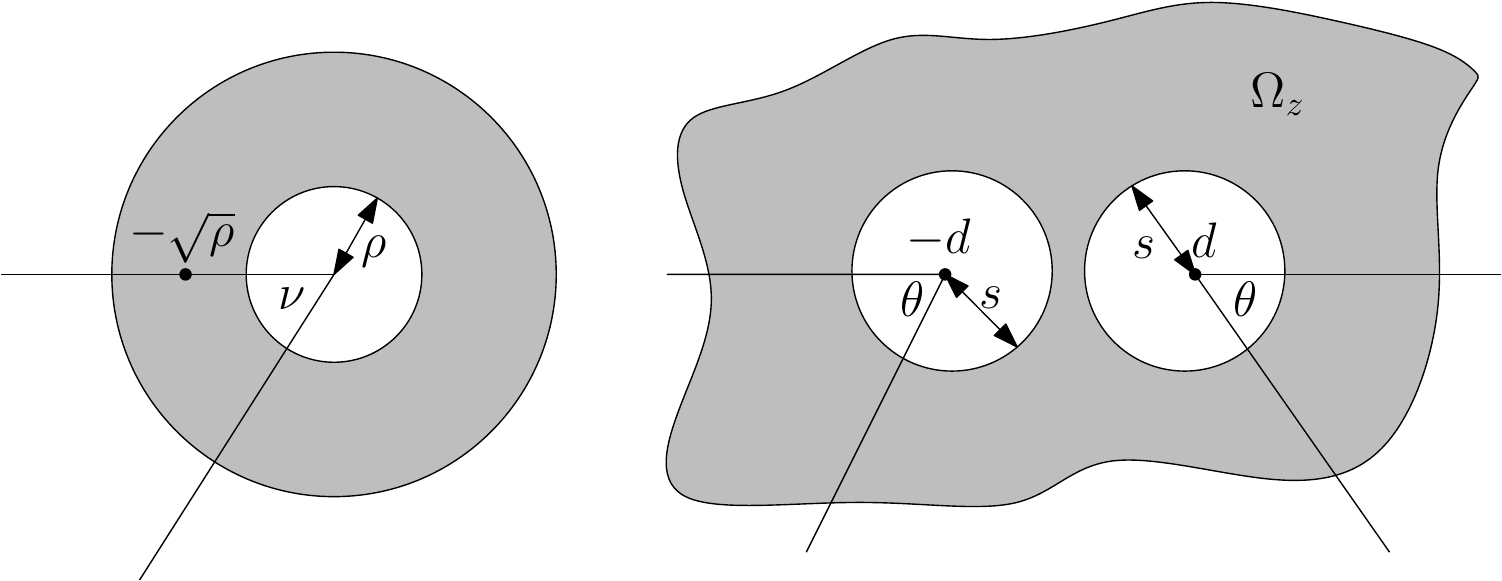}
\caption{Definition sketch for the angles $\nu$ and $\theta$ in (\ref{2.18.2})
and (\ref{2.18.2a}). The annulus $\rho < \zeta < 1$ and the region exterior
to the two discs centred at $z=\pm d$ are shown shaded.
\label{Fig11a}}
\end{center}
\end{figure}
 
\subsection{The function $\omega (z)$ and rapidly decaying series}

\begin{Definition}
Define 
\begin{equation}
\label{2.1}
\omega (z) = \frac{1}{\zeta z_\zeta} \left \{\zeta W_\zeta (\zeta (z)) 
-U_0 \zeta z_\zeta - \frac{A}{\pi T} (U_0 - {\bar U}_0 ) \right \} 
\end{equation}
\end{Definition}
\begin{Remark}
{\rm 
On use of (\ref{eq:eqEX}) and \eqref{2.1} we find
\begin{equation}
\label{eqRemark1}
\omega (z) = 
\frac{dW}{dz} - U_0 - \frac{A (U_0 - {\bar U}_0 )}{\pi T \zeta z_\zeta} =
\frac{dW}{dz} - U_0 + \frac{2 A^2 (U_0 - {\bar U}_0 )}{\pi T (z^2-A^2)}.
\end{equation}
It is important to note that the two points $z=\pm{A}$ lie outside the domain 
$\Omega_z:= 
\left \{ z: |z+d| \ge s, |z-d| \ge s \right \}$.
It follows that
$\omega $ is analytic in $\Omega_z$ with $\omega \rightarrow 0$ as $z \rightarrow \infty$. Therefore,
\begin{equation}
\label{2.5}
\omega (z) = \sum_{j=1}^\infty \frac{c_j s^j}{(z-d)^j}
+ \sum_{j=1}^\infty \frac{d_j s^j}{(z+d)^j} \ ,
\end{equation}
where 
\begin{equation}
\label{2.6}
c_j = 
- \frac{1}{2 \pi {\rm i} s^j} \oint_{|z'-d|=s} \omega (z') (z'-d)^{j-1} dz'
= \frac{1}{2 \pi} \int_{-\pi}^\pi \omega (d+s e^{-{\rm i}  \theta'} ) e^{-{\rm i}  j \theta'} d\theta'
\end{equation}
\begin{equation}
\label{2.7}
d_j = 
-\frac{1}{2 \pi {\rm i}  s^j} \oint_{|z'+d|=s} \omega (z') (z'+d)^{j-1} dz'
= \frac{(-1)^{j}}{2 \pi} \int_{-\pi}^\pi \omega (-d-s e^{{\rm i}  \theta'} ) e^{{\rm i}  j \theta'} d\theta'
\end{equation}
and where contour integration in the $z'$ plane is understood in a clockwise sense
}
\end{Remark}

On $z=d+s e^{-{\rm i}  \theta}$, corresponding to $\zeta = -e^{{\rm i}  \nu}$, 
(\ref{eqRemark1}) implies
\begin{multline}
\label{2.3}
\omega (z) = -U_0 + 
\frac{2 {\rm i} A U_0 \nu_\theta}{(z-d)} \left [ \partial_\nu K ( - \rho^{-1/2} \zeta) 
- \frac{{\rm i}}{2\pi T} \right ] \\  
-\frac{2 A {\rm i} {\bar U}_0 \nu_\theta}{(z-d)} 
\left [ \partial_\nu K ( - \rho^{1/2} \zeta) - \frac{{\rm i} }{2 \pi T}
\right ],
\end{multline}
and on $z=-d-se^{{\rm i} \theta}$, corresponding to $\zeta= -\rho e^{{\rm i} \nu}$, 
(\ref{eqRemark1}) implies
\begin{multline}
\label{2.4}
\omega (z) = -U_0 - 
\frac{2 {\rm i} A U_0 \nu_\theta}{(z+d)} 
\left [ \partial_\nu K ( - \rho^{-1/2} \zeta) 
- \frac{{\rm i}}{2\pi T} \right ] \\  
+\frac{2 {\rm i} A {\bar U}_0 \nu_\theta}{(z+d)} 
\left [ \partial_\nu K ( - \rho^{1/2} \zeta) - \frac{{\rm i} }{2 \pi T}
\right ].
\end{multline}
In particular integration in \eqref{2.6}-\eqref{2.7} for $j=1$ is explicit resulting in
\begin{equation}
\label{2.8.0}
s c_1 = \frac{A}{\pi T} (U_0-{\bar U}_0 )  = -s d_1.  
\end{equation}
Further, it is clear from (\ref{2.1}) that
\begin{equation}
\label{2.8}
W(z) = C \int^z \omega (z') dz' + 
\frac{A}{\pi T} (U_0-{\bar U}_0 ) \log \left ( \frac{z+A}{z-A} \right )  
\end{equation}
implying from (\ref{2.5}) and (\ref{2.8.0}) 
that for some constant $C$
\begin{equation}
\label{2.9}
W = C - \sum_{k=2}^\infty \frac{ s^k c_k}{(k-1) (z-d)^{k-1}}
- \sum_{k=2}^\infty \frac{ s^k d_k}{(k-1) (z+d)^{k-1}}
+ W_2,
\end{equation}
where
\begin{equation}
\label{2.10}
W_2 (z) = 
\frac{A}{\pi T} (U_0-{\bar U}_0 ) \log \left ( \frac{z-d}{z-A} \right )  
+\frac{A}{\pi T} (U_0-{\bar U}_0 ) \log \left ( \frac{z+A}{z+d} \right ).
\end{equation}

We will now prove that the Laurent series coefficients of $\omega(z)$, 
namely $\lbrace c_k \rbrace$, $\lbrace d_k \rbrace$, decay rapidly in 
$k$ uniformly for $T \in (0, T_0]$ in the following sense:

\begin{Proposition}
\label{Prop1}
For any integer $k \ge 0$,
\begin{equation}
\label{eqomega2}
\Big | j^k c_j \Big | \ ,  \Big | j^k d_j \Big | \le \frac{M_k}{T} 
\end{equation}   
where $M_k$ is independent of $T \in (0, T_0]$.
\end{Proposition}

\begin{Remark}
\label{RemProp1}
From (\ref{2.6}), (\ref{2.7}), (\ref{2.3}) and (\ref{2.4}), the 
proof of Proposition \ref{Prop1} will follow after we show that
on each of the boundaries $\zeta=-e^{{\rm i} \nu (\theta)}$ and 
$\zeta=-\rho e^{{\rm i} \nu (\theta)} $, 
for any integer $k \ge 0$
\begin{equation}
\label{eqRemProp1}
T \partial_\theta^{k+1} \left ( K (-\rho^{-1/2} \zeta) - \frac{{\rm i} \nu}{2 \pi T} 
\right ), \qquad   
T \partial_\theta^{k+1} \left ( K (-\rho^{1/2} \zeta) - \frac{{\rm i} \nu}{2 \pi T} 
\right )      
\end{equation}
are each $2 \pi$-periodic function of $\theta$ and have bounds $M_k$ independent of $T$.
Periodicity is clear since, from (\ref{3.2a}) and (\ref{3.4a}), both
\begin{equation}
{\partial \over \partial \theta} \left [K (-\rho^{-1/2} \zeta) - \frac{{\rm i} \nu}{2 \pi T} \right ], \qquad {\rm and} \qquad
{\partial \over \partial \theta} \left [
K (-\rho^{1/2} \zeta) - \frac{{\rm i} \nu}{2 \pi T} \right ]
\end{equation}
are obviously periodic in $\nu$ for $\zeta=-e^{{\rm i}\nu}, -\rho e^{{\rm i}\nu}$, and therefore in $\theta$; derivatives of $\nu$ are
also periodic in $\theta$ as is clear from 
(\ref{3.3}).
Hence, we only need to prove the bounds. 
\end{Remark}

\begin{Remark}
We will find the following identity on
derivativies of smooth composite functions useful:
\begin{equation}
\label{eqComp}
\partial_\theta^l F (\eta (\theta) ) =  
\sum_{j_m \ge 0 \atop \sum_{m=1}^l m j_m = l } 
C_{j_1,j_2, \cdots, j_l} \left ( \partial_\eta^{(\sum_{m=1}^l j_m)} F \right ) 
\left ( \prod_{m=1}^l \left [\partial_\theta^m \eta \right ]^{j_m} \right ),
\end{equation}
\end{Remark}
\begin{Lemma}
\label{lem1}
For $T \in (0, T_0]$, define 
\begin{equation}
\label{eqchi}
\chi = \exp \left [ -\frac{2}{T} \arctan \left ( \frac{1-\sqrt{\rho}}{1+\sqrt{\rho}} \tan \frac{\theta}{2} 
\right )
\right ]. 
\end{equation}
Then, for $\theta \in [0, \pi )$,
\begin{equation}
\label{eqchi2}
\Big | \partial_{\theta}^l \chi \Big | \le M_l 
\end{equation}
where constant $M_l$ is independent of $T$ and is only dependent on $l$.
\end{Lemma}

\begin{proof}
In the following proof, $C_l$ is a generic constant depending on $l$ that is allowed
to vary from step to step.
We introduce for convenience intermediate variables $x= \tan(\theta/2)$, 
$y = c_0 T x$ and $\nu_1 = (1/T) \arctan y$; it follows
from (\ref{eqchi}) that
\begin{equation}
\chi = e^{-2 \nu_1}.
\end{equation}
First, consider the case when $x T^{2/3} \le 1$, implying $ y$ is small. 
Taylor expansion of $\arctan y$ gives
\begin{equation}
\label{eqnu1x}
\nu_1 = c_0 x + \sum_{j=3}^\infty (-1)^j j T^{j-1} c_0^j \frac{x^j}{j}  
\end{equation}
We note that for any $\theta \in [0, \pi)$, $|\partial_\theta^j x | \le C (1+x)^{j+1}$.
Identifying $\eta$ with $x$ and $F$ with $c_0^j T^{j-1} x^{j}$ in \eqref{eqComp}
to determine bounds on the sum in (\ref{eqnu1x}), it follows 
that for $x T^{2/3} \le 1$,
\begin{equation}
\label{eqthetalnu1}
\Big | \partial_\theta^l \nu_1 \Big |
\le C_l (1+x)^{l+1} \left ( 1 +\sum_{j=3}^\infty T^{j-1} c_0^{j} x^{j} \right )
\le C_l (1+x)^{l+1} \le C_l (1+\nu_1)^{l+1} 
\end{equation}
Now identifying $\eta$ with $\nu_1$ and taking $F(\nu_1) = e^{-2\nu_1}$ in 
\eqref{eqComp} and using  
\eqref{eqthetalnu1},
it follows that
\begin{equation}
\Big | \partial_\theta^l e^{-2 \nu_1} \Big |
\le C_l e^{-2 \nu_1} (1 + \nu_1)^{2l} \le C_l 
\end{equation}
Now consider, $ T^{-2/3} \le x \le \frac{4}{c_0 T}$, {\it i.e.} $ c_0 T^{1/3} \le y \le 4$. 
We note that 
\begin{equation}
\label{eq:eq82.0}
\left | {\partial^l y \over \partial \theta^l} \right | 
= c_0 T \left | {\partial^l x \over \partial \theta^l} \right | 
\le C_l T (1+x)^{l+1} 
\end{equation}
Identifying $\eta$ with $y$, and taking $F(y) = \arctan (y)$
in \eqref{eqComp},
it follows from
\eqref{eq:eq82.0} that for $y \le 4$, 
\begin{equation}
\label{eqn1thetaj}
\Big | \partial_\theta^l \nu_1 \Big | \le 
\frac{C_l}{T} (1+x)^l  
\end{equation}
Now, identifying $\nu_1$ with $\eta$ and taking $F(\nu_1) =
e^{-2 \nu_1}$ in \eqref{eqComp}, and noting $y \ge c_0 T^{1/3}$ implies
$\nu_1 \ge \frac{1}{T} \arctan [c_0 T^{1/3} ]$, it follows from
\eqref{eqn1thetaj}: 
\begin{equation}
\Big | \partial_\theta^l e^{-2 \nu_1} \Big | \le \frac{C_l}{T^l} (1+x)^l 
e^{-2 \nu_1} 
\le C_l
\end{equation}
since $y \le 4$ implies $x^l \le \frac{4^l}{(c_0 T)^l}$.
For $y \ge 4$, note the the convergent series representation in powers of $\frac{1}{y}$
\begin{equation}
\nu_1 = \frac{1}{T} \left ( \frac{\pi}{2} -\sum_{j=0}^\infty \frac{(-1)^j}{(2j+1) y^{2j+1}}  
\right )
=: \frac{1}{T} f \left (\frac{1}{y} \right ) 
\end{equation}
We also note that
\begin{equation}
\label{eqinvythet}
\frac{1}{y} = \frac{1}{c_0 T} \cot \frac{\theta}{2} 
\end{equation} 
and the mapping $\theta \rightarrow T y^{-1}$ is smooth for $y \ge 4$.
Identifying $\frac{1}{y}$ with
$\eta$ and $F$ with $\frac{1}{T} f$ 
in \eqref{eqComp}, it follows from  
\eqref{eqinvythet} that 
\begin{equation}
\Big | \partial_\theta^l \nu_1 \Big | \le \frac{C_l}{T^{l+1}}  
\end{equation}
On the otherhand if $ y \ge 4$, then $e^{-2 \nu_1} \le e^{-\frac{2}{T} \arctan 4} $ and therefore,
using \eqref{eqComp} again
\begin{equation}
\Big | \partial_\theta^l e^{-2 \nu_1} \Big | \le \frac{C_l e^{-\frac{2}{T} \arctan 4}}{T^{2l}}
\le C_l
\end{equation}
Hence the lemma follows for all cases.
\end{proof}
\begin{Lemma}
\label{lem2}
For $T \in (0, T_0]$, 
define 
\begin{equation}
F_0(\nu) = \frac{e^{-\nu/T}}{1 \pm{{\rm i}} e^{-\nu/T} }, \qquad
F_1 (\nu) = \frac{\mu e^{\pm{\nu}/T}}{1+{\rm i} \mu^m e^{\pm{\nu}/T}}, \qquad
F_2 (\nu) = \frac{\mu e^{\pm{\nu}/T}}{1-{\rm i} \mu^m e^{\pm{\nu}/T}}.
\nonumber 
\end{equation}
For $j=0,1,2$, for any integer $l \ge 1$, there exists constant $M_l$
independent of $T$ so that for  
$\nu \in \left (-\pi, \pi \right ) $, 
\begin{equation}
\Big | \partial_\theta^l F_j (\nu) \Big | \le M_l.
\end{equation}
\end{Lemma}
\begin{proof}
Consider first $F_0$. Since use of the transformation 
$\nu \rightarrow -\nu$ in the 
alternate form 
\begin{equation}
F_0 = \mp{{\rm i}} - \frac{e^{\nu/T}}{1\mp{{\rm i}} e^{\nu/T}}
\end{equation}
makes the arguments
for $\nu \in \left ( -\pi, 0 \right ]$ equivalent to those for $\nu \in \left [0, \pi \right )$
then
we restrict to the latter case.
Since $F_0$ is a rational function of $\chi=e^{-\nu/T}$
with singularity location in the complex plane at $O(1)$ distance, 
and with all derivatives with respect
to $\chi$ bounded for $\chi \in [0, \infty )$,
it suffices to show that $\Big | \partial_\theta^l \chi \Big | \le C_l$, with $C_l$ independent of
$T$. But this has been proved in Lemma \ref{lem1}.
The argument for each of $F_1$ and $F_2$ is similar, except that 
we need to use
$$ \Big | \mu e^{\nu_1} \partial^{k}_\theta \nu_1 \Big | 
\le C_0 e^{-2\pi/T} \nu_1^{k+1} e^{\nu_1}  \le M_k, 
$$
independent of $T$.
\end{proof}
\noindent{\bf Proof of Proposition \ref{Prop1}:}
From Remark 
\ref{RemProp1}, it is enough to obtain bounds independent of $T$ for terms appearing in  
\ref{eqRemProp1} on each circle $\zeta= -e^{{\rm i} \nu (\theta)}$ and $\zeta= -\rho e^{{\rm i} \nu (\theta)}$.
However, from (\ref{3.2a}) and \eqref{3.4a}, it is clear from applying Lemma
\ref{lem2} on each term that this is true and the proof of Proposition \ref{Prop1} follows after
use of dominated
convergence theorem to justify the commutation of infinite sum with derivatives in $\theta$.      

\begin{Remark}
\label{RemThm1.2}
Note that the second part of Theorem \ref{Thm1} on the logarithmic decomposition
is now obvious from Proposition \ref{Prop1} and on use of \eqref{2.9} and \eqref{2.10}.
\end{Remark}

\subsection{Series representation of $W_2$}

Note from (\ref{2.10}) that
\begin{equation}
\label{3.1}
W_2 (z) = -\frac{A}{\pi T} \left (U_0-{\bar U}_0 \right ) W_{2,1} (z) 
+ \frac{A}{\pi T} \left (U_0 - {\bar U}_0 \right )  W_{2,2} (z)
\end{equation}
where
\begin{equation} 
\label{3.2}
W_{2,1} (z) = \log \left ( \frac{z-A}{z-d} \right ) = \log \left ( 1 + \frac{\sqrt{\rho} s}{z-d} \right )   
\end{equation}
\begin{equation}
\label{3.3.0}
W_{2,2} (z) = 
\log \left ( \frac{z+A}{z+d} \right ) = \log \left ( 1 - \frac{\sqrt{\rho} s}{z+d} \right )  
\end{equation}
We note that $W_{2,1}$ is an analytic function outside the circle $|z-d|=s$, whose Laurent series
in powers of ${s}/({z-d})$ has a radius of convergence
$\rho^{-1/2}$, while $W_{2,2}$ is an analytic function outside $|z+d|=s$, whose Laurent series 
in powers of ${s}/({z+d})$ also has a radius of convergence
$\rho^{-1/2}$. If we use (\ref{eq:eq2})-(\ref{eq:eq3}), then it is clear that
\begin{equation}
\label{3.4}
W_{2,1} (Z (\zeta)) 
= -\log \left ( \frac{d+A}{2A} \right )
- \log \left ( 1 +\rho^{1/2} \zeta \right )  
\end{equation}
The radius of convergence in powers of $\zeta$ 
for the analytic function $W_{2,1} (Z (\zeta))$ inside the
unit circle is again $\rho^{-1/2} $ and there is no advantage using the
$\zeta$ scheme versus the $z$ scheme.
Also, 
\begin{equation}
\label{3.4}
W_{2,2} (Z (\zeta)) 
= -\log \left ( \frac{d+A}{2A} \right )
- \log \left ( 1 +\rho^{3/2} \zeta^{-1} \right )  
\end{equation}
and, again, the radius of convergence of a series in $\rho/\zeta$ of this analytic function in $|\zeta| > \rho$ is
$\rho^{-1/2}$ 
and there is no advantage using the $\zeta$ scheme over the $z$ scheme. When $\rho$ is very
close to 1, the geometric decay factor is poor and one needs a large number of
terms in the series represention for each of $W_{2,1}$ and $W_{2,2}$. Indeed, note that 
on the circle $\zeta=-e^{i \nu}$,
\begin{equation}
\label{3.5.1}
\Big | \partial_\nu^k W_{2,1} \left ( Z (-e^{-i \nu} ) \right )  \Big | \le C_0 T^{-k} \ ,
\end{equation} 
which is reflected in the Taylor series coefficient of $W_{2,1}  (\zeta) = \sum_{j=0}^\infty w_j \zeta^j$ in the 
observation
\begin{equation}
\label{3.6}
j^k |w_j| = \frac{j^k}{j} \rho^{1/2 j} = j^{k-1} e^{-\pi T j/2} \le C_0 T^{-(k-1)}.
\end{equation}
From the logarithmic expansion in (\ref{3.2}) in powers of $s/(z-d)$, i.e.,
$W_{2,1} =\sum_{k=1}^\infty \frac{W_j s^j}{(z-d)^j} $, we also have
\begin{equation}
\label{3.6.1}
j^k |W_j| \le C_0 T^{-(k-1)}.
\end{equation}
Similar statements 
hold for 
series expansions of $W_{2,2}$ in powers of $\rho/\zeta$ or
$s/(z+d)$.

We now seek to prove that the hybrid basis scheme
does
better than (\ref{3.6}) {or (\ref{3.6.1})}. 
We demonstrate this only for $W_{2,1}$ 
since an analogous construction is possible for $W_{2,2}$, with
$s/(z-d)$ replaced by $s/(z+d)$ and $\zeta$ replaced by $\rho/\zeta$ in the series representation.

\subsection{Hybrid decomposition of $W_{2,1}$}

To be more precise, we prove the following
theorem:
\begin{Proposition}
\label{prop3}
There exists  decomposition of 
\begin{equation}
W_{2,1} = \omega_1 (\zeta) + \omega_2 \left ( \frac{s}{z-d} \right ) ~\ ,~ 
W_{2,2} = \omega_3 (\rho \zeta^{-1}) + \omega_4 \left ( \frac{s}{z+d} \right ) \ , 
\end{equation}
where each of $\omega_1$, $\omega_2$, $\omega_3$ and $\omega_4$ are analytic when
its argument is inside the unit circle and for any $m=1,2,3,4$,
\begin{equation}
\omega_m (\eta) =  
\sum_{j=0}^\infty a_{j,m} \eta^j 
\end{equation}
such that for any integer $k \ge 1$, 
\begin{equation}
\Big | j^k a_{j,m} \Big | \  
\le \frac{M_k}{T^{k/2}},
\end{equation}
where $M_k$ is independent of $j$ and $T$, ({\it i.e.} of $\rho$)).
\end{Proposition}
{The proof} of Proposition \ref{prop3} 
will await preliminary Lemmas that pertain to 
construction of the decomposition of $W_{2,1}$ into $\omega_1$ and $\omega_2$ and 
smoothness properties of each the $\nu$ and $\theta$ representations
on the circular boundaries and uniform control of all derivatives 
as $T \rightarrow 0^+$ ($\rho \rightarrow 1^-$). The arguments are nearly identical
for $W_{2,2}$.

\subsection{Construction of $\omega_1$, $\omega_2$,  Proof of Proposition \ref{prop3} and Theorem 1}

\begin{Definition}
\label{Def1}
We choose $\delta$ that shrinks with $T$ as $T \rightarrow 0$ so as to satisfy
$1 >> \delta >> T^{1-\epsilon_1} $ for some $1 > \epsilon_1 > 0$ independent of $T$. A more precise choice
will be made later to optimize the decay rate of power series. 
We define an even, smooth, cut-off function 
$\Phi \in {\bf C}^\infty [-\pi, \pi]$ so that
\begin{equation}
\label{4.0}
\Phi (\nu) = 1  ~~{\rm for} ~ |\nu| \le \delta \ , \Phi (\nu) = 0 ~~{\rm for} ~|\nu| \in [2 \delta, \pi] 
\end{equation}
and with property $\Big | \partial_\nu^l \Phi (\nu) \Big | \lesssim \delta^{-l}$. There are standard choices
for such functions (see Evans PDE book for instance).
Corresponding to such a choice of $\Phi$, we define a  
single valued analytic function $\omega_{1} (\zeta)$ analytic in $ |\zeta| < 1$ such that
on the boundary $|\zeta|=1$ 
with representation $\zeta=-e^{{\rm i} \nu}$, $\nu \in [-\pi, \pi]$ 
\begin{equation}
\label{4.1}
\Re ~ \omega_1 (- e^{{\rm i} \nu} ) = 
\left (1 - \Phi (\nu) \right ) \Re ~W_{2,1} \left (Z (-e^{{\rm i} \nu} ) \right )
\end{equation}
We define an analytic function $\omega_2 (\zeta)$ analytic in $|\zeta|< 1$ so that  
on 
$\zeta=-e^{{\rm i}\nu}$ 
\begin{equation}
\label{4.5}
\Re ~\omega_2 (- e^{{\rm i} \nu} ) = 
\Phi (\nu) ~\Re ~W_{2,1} \left ( Z (- e^{{\rm i} \nu})  \right ) 
\end{equation}
\end{Definition}
\begin{Lemma}
\label{lem0}
Define $H(\nu) =  \left (1 -\Phi (\nu) \right ) 
\Re ~W_{2,1} \left (Z (- e^{{\rm i} \nu} ) \right )$.
$H$ may be extended to be a smooth $2\pi$-periodic function of $\nu$ satisfying 
\begin{equation}
\label{eqlem0}
\Big | \partial_{\nu}^k H (\nu )  \Big | \le \frac{M_k}{\delta^k},
\end{equation}
where $M_k$ may be chosen independent of $\delta$ and $T$ ($\rho$). 
\end{Lemma}

\begin{proof}
Since $1-\Phi$ is a smooth cut-off function with support in $\pi \ge |\nu| \ge \delta$ and $1-\Phi =1$ for
$\pi \ge |\nu| \ge 2 \delta$, the $2\pi$ periodicity and smoothness of $\log (1-\rho^{-1/2} e^{i \nu} )$ 
translates into smoothness of $H$ with $\partial_\nu^k H(\pi) = \partial_\nu^k H(-\pi)$ for any integer $k \ge 0$. 
We only need to show the bounds. 
We note that 
$\Big | \delta^j \partial_\nu^j \Phi \Big | \le c_0$, independent of $T$ and $\delta$, and for $|\nu| \in [\delta, \pi]$, 
for $k \ge 1$,
\begin{equation}
\delta^k \Big |  
\partial_\nu^k \log \left ( 1-\rho^{-1/2} e^{i \nu} \right ) \Big |
\le C_0 
\end{equation}
Therefore, from a Leibnitz representation of the derivatives of a product, the Lemma follows immediately.
\end{proof}
\begin{Proposition}
\label{prop0}
$\omega_1$ defined in Definition \ref{Def1} with representation
\begin{equation}
\label{eqom1lem}
\omega_1 (\zeta) = \sum_{j=0}^\infty a_j \zeta^j 
\end{equation}
satisfies
\begin{equation}
j^k |a_j| \le M_k/\delta^k 
\end{equation}
for any $k \ge 1$, where $M_k$ is independent of $\delta$ and $T$.
\end{Proposition}

\begin{proof}
We simply note that on the boundary $\zeta=-e^{{\rm i} \nu}$, 
from the Fourier representation and noting that $H$ is even we have
\begin{equation}
\label{eqomega1}
\Re ~\omega_1 \left (-e^{{\rm i} \nu} \right ) = H(\nu) = \sum_{j=0}^\infty (-1)^j a_j \cos (j \nu),
\end{equation}
where previous Lemma implies
\begin{equation}
j^k |a_j| \le M_k/\delta^k.
\end{equation}
Hence the Lemma follows since  
\eqref{eqomega1} implies 
\eqref{eqom1lem}
because $\Im[\omega_1 (0)]=0$. 
\end{proof}
We now consider the problem of representing $\omega_2$. 
Note that on $\zeta=-e^{{\rm i} \nu}$
\begin{equation}
\Re ~\omega_2 \left (\boldsymbol{\zeta} (d+s e^{-{\rm i} \theta}) \right )
= \Phi (\nu (\theta) ) \log \left ( 1 + \rho^{-1/2} e^{{\rm i} \theta} \right ).
\end{equation}
Define
\begin{equation}
\theta_m = 2 \arctan \left \{ \left ( \frac{1+\sqrt{\rho}}{1-\sqrt{\rho}} \right ) \tan \delta \right \}. 
\end{equation}
From relation (\ref{3.3}) between $\nu$ and $\theta$,
$\Phi (\nu (\theta)) =0 $ for $|\theta| > \theta_m$. Also, $1  \gg \delta \gg T$ implies that
\begin{equation}
\label{eqthetam}
\theta_m = \pi - \frac{2}{\tan \delta}  \left ( \frac{1-\sqrt{\rho}}{1+\sqrt{\rho}} \right ) + O (T^3/\delta^3).
\end{equation}

\begin{Lemma}
\label{lemnu}
Consider the change of variable $\nu =\nu (\theta)$ defined by \eqref{3.3}. 
For $\nu \in [-2 \delta, 2\delta ]$, {\it i.e.} $\theta \in [-\theta_m, \theta_m]$, where
there exists constant $C_0$ independent of $T$ and $\delta$ so that 
\begin{equation}
\Big | \partial_\theta^k  \nu \Big |  \le C_0 \frac{\delta^{k+1}}{T^k}.
\end{equation}
\end{Lemma}

\begin{proof}
It is convenient to introduce the intermediate variables 
\begin{equation}
x= \tan \left (\frac{\theta}{2}\right ), \qquad y = \left ( \frac{1+\sqrt{\rho}}{1-\sqrt{\rho}} \right ) x.
\end{equation}
 Then $\nu = 2 \arctan y$. 
Since $\delta$ is small, it follows that $y$, which is at most $O(\delta)$ is also small 
for $\theta \in [-\theta_m, \theta_m]$. 
It immediately follows that 
all derivatives of $\nu$ with respect to $y$ are bounded independent of any parameter. 
Now consider $x=x(\theta)) $. Clearly derivatives of $x(\theta)$ can only be large for 
$|\theta|$ near $\pi$, where there is a simple pole. It follows that $\Big | \partial_\theta^l x \Big | \le c_0 x^{l+1}$.  
Therefore, $\Big | \partial_\theta^k y \Big | \le c_0 T x^{k+1} \le c_0 y^{k+1}/T^{k} $. Since $y = O (\delta)$,
the Lemma immediately
follows.
\end{proof}

\begin{Lemma}
\label{lemPhi}
For the cut-off function $\Phi$ as defined earlier, with $\theta \in [-\theta_m, \theta_m]$,
for any integer $k$, there exists constant $c_0$ independent of $\delta$ and $T$ so that
\begin{equation}
\Big | \partial_\theta^k  \Phi (\nu (\theta) ) \Big | \le c_0 \frac{\delta^k}{T^k}.
\end{equation}
\end{Lemma}
\begin{proof}
We simply make use of (\ref{eqComp}) 
and invoke the previous Lemma and 
the bounds  $\Big | \partial_\nu^l \Phi \Big | \le C \delta^{-l}$.  
\end{proof}
\begin{Lemma}
\begin{equation}
H_2 (\theta) = \Re ~\omega_2 \left (-e^{i \nu (\theta)} \right )
=\Phi (\nu (\theta) ) \Re ~\log \left ( 1 + \rho^{1/2} e^{i \theta} \right ) 
\end{equation}
Then for any $k \ge 1$, for $\theta \in [-\pi, \pi ]$,  
\begin{equation}
\Big | \partial_\theta^k H_2 (\theta) \Big | \le c_0 \frac{\delta^k}{T^k}     
\end{equation}
\end{Lemma}
\begin{proof}
We note that since the support of $\Phi (\nu (\theta) )$ is $|\theta| \le \theta_m$ and
$\pi -\theta_m = O \left (\frac{T}{\delta} \right )$, and the observation that for $\theta \in [-\theta_m, \theta_m]$,
$$ \Big | \partial_\theta^k \log \left ( 1 + \rho^{1/2} e^{i \theta} \right ) \Big |
\le M_k \frac{\delta^k}{T^k} $$   
for $M_k $ independent of $T$ and $\delta$.
Using the Leibnitz rule, 
and Lemma \ref{lemPhi}, the present Lemma follows.
\end{proof}

\begin{Proposition}
\label{prop6}
The analytic function $\omega_2 $ in $|\zeta| < 1$, {\it i.e.} $|z-d| > s$ has the representation
\begin{equation}
\omega_2 \left ( \boldsymbol{\zeta} (z) \right ) = \sum_{j=0}^\infty \frac{b_j s^j}{(z-d)^j}  \ ,
\end{equation}  
where for any $k \ge 1$,
\begin{equation}
\label{bjkbound}
j^k |b_j | \le \frac{M_0 \delta^k}{T^k}.
\end{equation}
\end{Proposition}
\begin{proof}
We use the fact that $\Re ~\omega \left ( \boldsymbol{\zeta} (z) \right ) = H_2 (\theta)$ and
from the previous Lemma  
\begin{equation}
H_2 (\theta) = \sum_{j=0}^\infty b_j \cos (j \theta ) 
\end{equation}
where $b_j$ satisfies \eqref{bjkbound}. Therefore, the proposition immediately follows.
\end{proof}

\subsection{Proofs of Proposition \ref{prop3} and Theorem 1 \label{proofs}}

The proof of Proposition \ref{prop3}
follows from Propositions \ref{prop0} and \ref{prop6} if we choose $\delta=T^{1/2}$ and
note that from the definition of 
$\omega_1$ and $\omega_2$, on $|\zeta|=1$, $\Re~ W_{2,1} = \Re ~\omega_1 + \Re ~\omega_2$ and 
therefore for $W_{2,1} = \omega_1 + \omega_2$ since the ambiguity in the imaginary constant is
resolved by the condition $\Im W_{2,1} (0) =0$. 

The argument for $W_{2,2}$ 
similar except for replacing $\zeta$ by
$\rho/\zeta$ and $s/(z-d)$ by $s/(z+d)$ in the arguments given already in the decomposition
of $W_{2,1}$.

\noindent{\bf Proof of Theorem \ref{Thm1}} follows by using Proposition \ref{prop3}
and the prior observation on the series decay for 
$W_{2,1}$ when a series just in $\zeta$ (or $s/(z-d)$) is used 
({see \eqref{3.6}-\eqref{3.6.1}}). 
Similar statement hold for $W_{2,2}$ in powers either in $\rho/\zeta$ or (or $s/(z+d)$).


\section{Connections with other methods}

It is natural to enquire as to the connection of this method with the rather different 
scheme proposed by Cheng and Greengard \cite{CG1}.
For the two disc
problem,
in order  to attain accuracy of ${\mathcal O}(10^{-6})$,
Cheng and Greengard \cite{CG1}
use a representation of the potential fields
that involves a sum of multipole expansions about
14,000 reflection points -- or ``images'' -- inside the discs. 
Despite this large number of multipole contributions, the total
number of unknowns associated with each disc is kept small
because (analytically) known reflection operators are used to generate
successive generations of multipole expansions given only
the parent coefficients.
These reflection operators are, however, highly specific
to the particular problem being solved there and are not generally
applicable
to other problems. We discuss this again later.
In principle, successive reflections between the two discs
produces an infinite number of possible image singularities
about which, following Cheng \& Greengard \cite{CG1}, one
could introduce multipole expansions to improve accuracy.
These image points tend to definite limit points.
Explicit formulae for these
limit points are given in equation (17) of \cite{CG1}:
if $z_1$ and $z_2$ denote the complex positions of the
disc centres then the two limit points are
\begin{equation}
z_1(\infty) = {z_1 + z_2 \over 2} -
\sqrt{s \hat d+\hat d^2/4} \left ({z_2 - z_1 \over |z_2-z_1|} \right ),
\label{eq:eqLIM}
\end{equation}
where $s$ is the disc radius and $\hat d=|z_2-z_1|-2s$ is the
distance between the two discs.
Careful
inspection of
(\ref{eq:eqINV}) and (\ref{eq:eqM2}) 
reveals that the method we have introduced here is essentially
equivalent to the sum of just four multipole expansions -- two about
the centres of the discs and two more about the points $\pm A$.
Use of (\ref{eq:eqLIM}) then reveals that 
our hybrid method corresponds to using multipole expansions
about the two centres of the discs {\em and}
their limit points after infinitely many
reflections.
Figure \ref{Fig1} illustrates the connection between the two methods
schematically in the case of the two-disc problem.

The evidence here suggests that, from a numerical standpoint,
it is satisfactory to use 
multipole expansions
only about the disc centres {\em and} the limiting image points of these
centres rather than using increasingly many
multipole expansions about successive images.

\begin{figure}
\begin{center}
\includegraphics[scale=0.4]{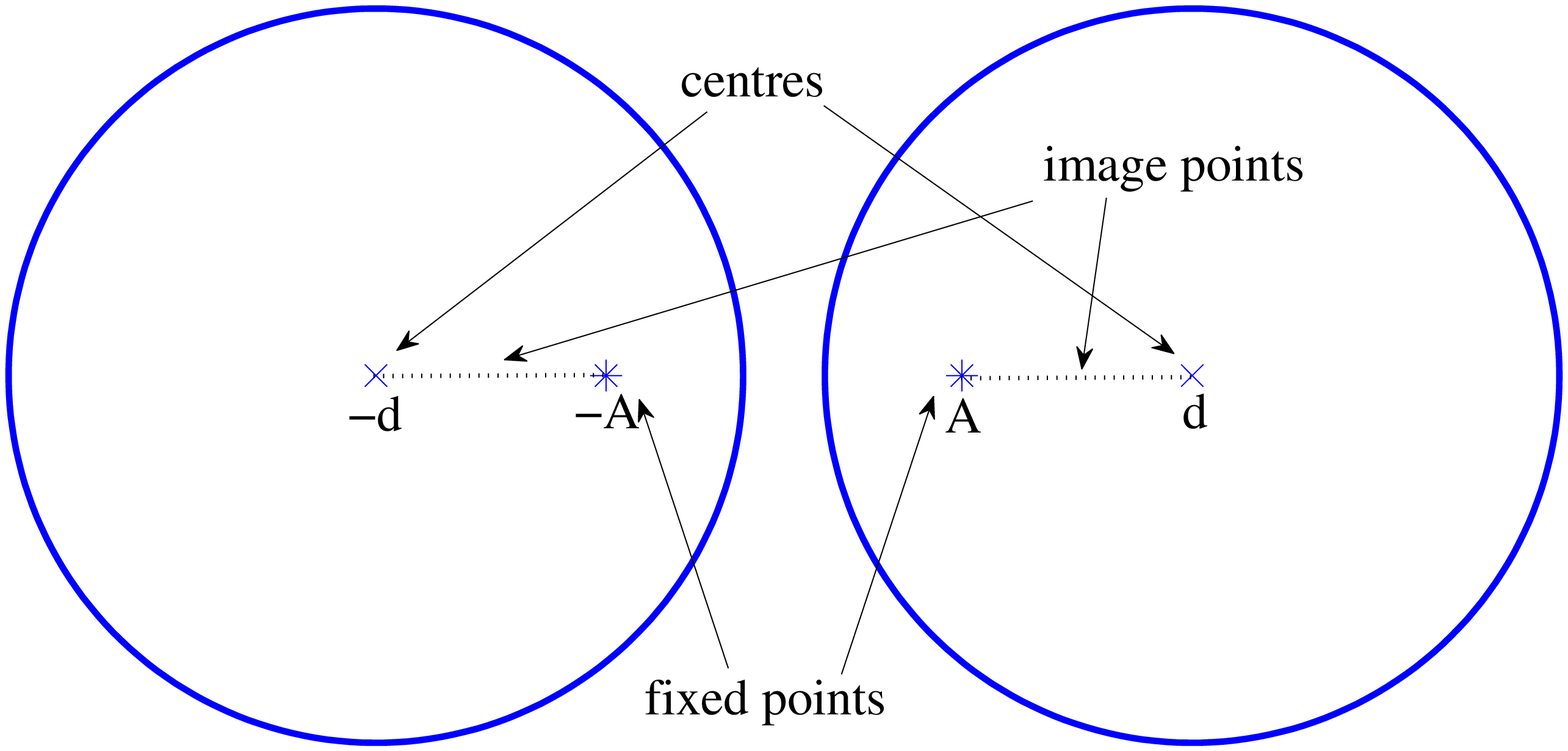}
\includegraphics[scale=0.4]{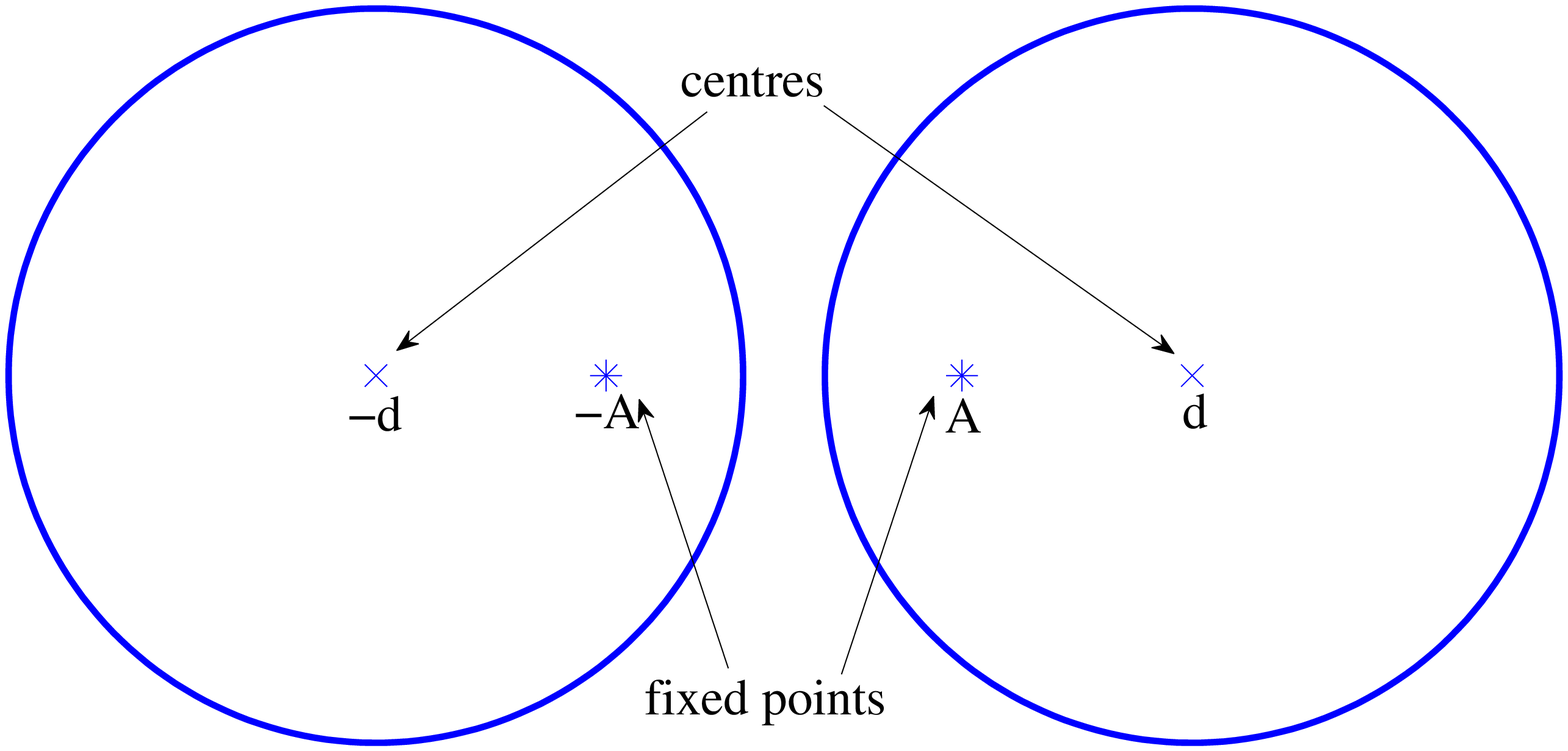}
\caption{Schematic illustrating that 
the method of Cheng \& Greengard \cite{CG1} (upper figure)
incorporates multipole expansions
about successive reflections (``images'') of the circle centres in the circle boundaries.
The new method (lower figure)
effectively
includes only multipole expansions about the circle
centres and the limit points of the  reflections. \label{Fig1}}
\end{center}
\end{figure}

\section{Extension to any number of discs}

Our essential idea is readily extendible to the multi-disc case.
 Suppose that discs $D_i$ and $D_j$
are separated by less than
some threshold value $\delta$. 
Based on the two-disc problem just analyzed, and depending on
the required accuracy and speed, an
appropriate $\delta$ might be of the 
order ${\mathcal O}(10^{-2})$, for example.
The
conformal mapping from a concentric
annulus $\rho_{ij} < |\zeta| < 1$ (for some 
real parameter $0 < \rho_{ij} < 1$) to the exterior of these two discs
is then found analytically. 
This is a M\"obius
transformation, as is its inverse mapping, which we will denote by
$\zeta_{ij}(z)$.
Then, in addition to Fourier-Laurent expansions about the disc centres,
terms of the form
\begin{equation}
 \sum_{k=1}^{\mathcal N} c_k^{(ij)}
  [\zeta_{ij}(z)]^k
  + \sum_{k=1}^{\mathcal N} {d_k ^{(ij)}\rho_{ij}^k \over [\zeta_{ij}(z)]^{k}}
 \label{eq:eqM2a}
\end{equation}
are also included (additively) in the representation of the complex potential function $w(z)$.
Equations for the
additional 
coefficients $\lbrace c_k^{(ij)}, d_k^{(ij)} | k=1,..., {\mathcal N} \rbrace$
are obtained, as before, by substituting
the expression for $w(z)$ into the boundary conditions
and evaluating at 
equi-spaced collocation points on the two circles
$|\zeta|=1, \rho_{ij}$.

As a benchmark test of our method in the multi-disc case, we revisit an example considered
by Cheng \& Greengard \cite{CG1}
involving a square array of nine equal discs.
Actually, the latter authors solve a two-phase problem
including a computation of the electrostatic
fields inside the discs as well as outside but
the conductivity ratio involved is $10^{8}$ 
which is so large that we
expect to be able to
retrieve their results using our method (the boundary value 
problem we solve corresponds
to the case of infinite conductivity ratio).
We have found that the hybrid basis scheme 
compares favourably with that of
\cite{CG1}.
It is capable of retrieving (with good
accuracy and efficiency) the same results 
in the most singular cases they analyze, including disc separations
as small as $10^{-7}$.

\begin{center}
\begin{tabular}{|p{1.in}|p{2in}|p{.8in}|} \hline
separation &  magnitude of dipole moment & modes ${\mathcal N}$ \\ \hline
$10^{-2}$  & 0.39194 & 5
\\ \hline
$10^{-3}$  & 0.43722 & 8
\\ \hline
$10^{-4}$  & 0.44964 & 15
\\ \hline
$10^{-5}$  & 0.45337 & 20
\\ \hline
$10^{-6}$  & 0.45453 & 25
\\ \hline
$10^{-7}$ & 0.45490 & 35
\\ \hline
\end{tabular}
\vskip 0.1truein
Table 6: Performance of hybrid basis scheme: the number of
modes required to retrieve the dipole moments, correct to ${\mathcal O}(10^{-4})$, found in
\cite{CG1} are recorded.
\end{center}

In the nine-disc problem,
the representation of the solution within the hybrid basis
scheme takes the form
\begin{equation}
w(z) = U_0 z + C +
\sum_{j=1}^9 \sum_{k=1}^{\mathcal N} {a_k^{(j)} s^k \over (z-c_j)^k}
+ \sum_{j=1}^{12} \sum_{k=1}^{\mathcal N}
b_k^{(j)} [\zeta_j(z)]^k + c_k^{(j)} [\zeta_j(z)]^{-k}
\label{eq:eqM5}
\end{equation}
where the coefficients
$\lbrace a_k^{(j)}, b_k^{(j)}, c_k^{(j)} \rbrace$ are to be found.
The set of disc centres is denoted $\lbrace c_j | j=1,...,9 \rbrace$.
We centred  
the central disc
at the origin, with four discs centred at $\pm d$ and $\pm {\rm i}d$ 
and four at $\pm d \pm {\rm i}d$.
Figure \ref{Fig9} shows a schematic.
The radius of each disc is $s=d-\epsilon/2$ where $\epsilon$
gives a measure of the disc separation.
To retrieve the results of Cheng \& Greengard we have
taken $d=0.2$.
The first double sum on the right hand side is the sum of the
multipole expansions about the 9 disc centres; the second sum
represents the additional terms arising from the 12 pairs of close-to-touching
discs. Only discs separated by exactly $\epsilon$ are included and, by inspection
of the geometry, it is easy to see that there
are 12 of these. In Figure \ref{Fig9} we have indicated
these interactions with double arrows.

\begin{center}
\begin{figure}
\begin{center}
\includegraphics[scale=0.6]{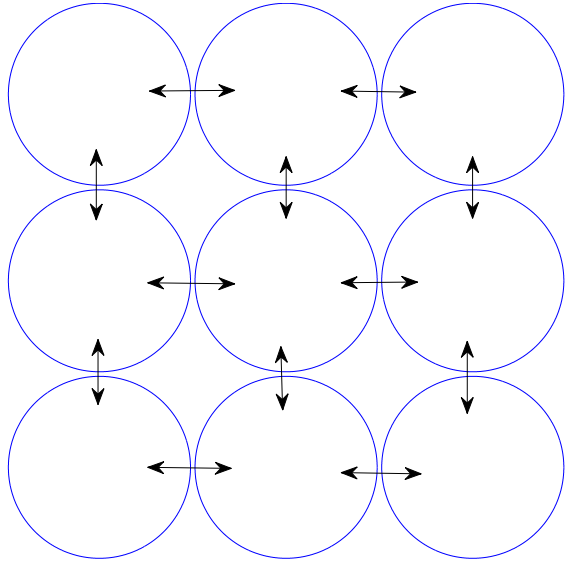}
\end{center}
\caption{The 9-disc example. The 12 close-to-touching disc interactions
are indicated by double arrows. For each such interaction, an additional
series term in powers of $\zeta_j$  is included in the representation (\ref{eq:eqM5}) for
$w(z)$. \label{Fig9}}
\end{figure}
\end{center}

Table 6 shows our results as a function of
disc separation $\epsilon$.  
We truncated the infinite
sums at ${\mathcal N}$ modes and took $2 {\mathcal N}+10$ collocation
points on each circle in order to yield an over-determined linear
system. 
Note that it is a simple matter to establish an expression for the
required dipole moment
from the representation (\ref{eq:eqM5}).


As a comparison of performance, 
Figure \ref{Fig4} shows
a graph of the logarithm of the disc separation against the number of modes
needed to attain the correct dipole moment correct to ${\mathcal O}(10^{-4})$.
Results are shown both for the new method given here as well as that of
Cheng \& Greengard \cite{CG1}. The key feature is that both graphs are close to
{\em linear} (albeit with different slopes). 
This means that both methods require just
a linear increase in the degree of truncation
as the order of magnitude of the separation is decreased. 
This is a major improvement on previous methods (e.g., the original
Fourier-Laurent method) where the required number of
modes appears to increase by an order of magnitude
as the separation decreases by
an order of magnitude and, therefore, rapidly become unviable for very small
separations.

\begin{center}
\begin{figure}
\includegraphics[scale=0.5]{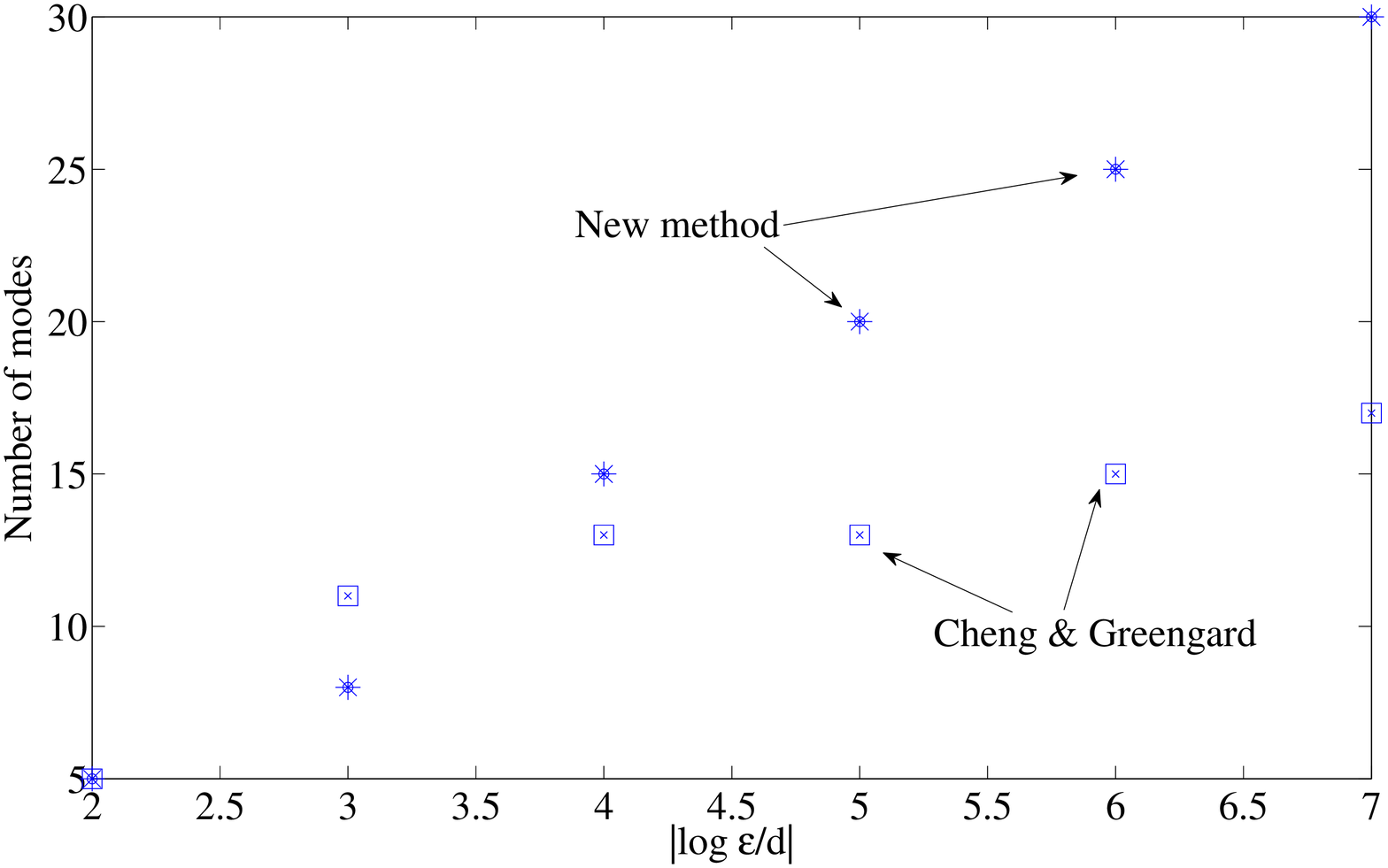}
\caption{Graph showing the number of modes needed (for accuracy
${\mathcal O}(10^{-4})$ in the 9-disc example) 
against the logarithm of the disc separation
$\log_{10}(\epsilon)$. Both methods require only a linear increase
in the number of modes as the separation decreases by an order
of magnitude. \label{Fig4}}
\end{figure}
\end{center}

\section{Discussion}

A new ``hybrid basis scheme'' has been presented to compute analytic
functions exterior to a collection of close-to-touching discs and satisfying
certain boundary conditions of the boundaries of those discs.
The method is conceptually simple and it is easy to implement.
It
affords similar advantages to the method
of Cheng \& Greengard \cite{CG1}
in that it
requires just a linear increase in the truncation as the
disc separation decreases by an order of magnitude.
The new method, however,
has the advantage of not requiring analytical knowledge of any
``reflection operators'' to produce the coefficients of multipole expansions
about 
higher generations of reflections.
This will be important
for problems involving the computation
of functions analytic outside a collection of close-to-touching discs
when the boundary value problems 
determining these functions
are such that
analogous reflection operators
are not known analytically.
It should also be clear that close-to-touching discs having unequal radii
are also amenable to the same method.


One motivation for seeking a simple and effective numerical
scheme for this problem
is our desire to
optimize a numerical scheme presented in \cite{Cro22}
for the computation of a special transcendental function known as the Schottky-Klein
prime function \cite{Baker}. 
This function is very useful \cite{Cro1, Cro2, Cro3} for finding analytical solutions
to problems involving multiply connected domains (indeed the function
$P(\zeta,\rho)$ used to find the exact solution (\ref{eq:eqEX}) is closely related
to a special
case of a Schottky-Klein prime function).
This function is defined in a multiply connected circular
region -- one whose boundaries consist purely of circles -- and, in certain applications, those circles draw close
together.
The hybrid basis scheme 
is expected to provide a viable means of accurately computing 
the Schottky-Klein prime function 
in such cases. Work on this is currently in progress.

\vskip 0.3truein \noindent {\bf Acknowledgments:} 
DGC acknowledges support from
an EPSRC Established Career Fellowship.
ST acknowledges support from an EPSRC Platform Grant held at Imperial College
London and support from the NSF DMS-110894.
DGC acknowledges many useful discussions with J. S. Marshall.

\section{Appendix: Alternative Representation of $K (\zeta)$}

Unfortunately (\ref{1.10}) is not in a suitable form 
to study its asymptotics as $\rho \rightarrow 1^{-}$. 
To find an alternative form it is useful to choose 
\begin{equation}
\label{A2.1}
\xi = \frac{1}{\pi} \log \zeta.
\end{equation}
Then it is clear from the representation  (\ref{1.10}) 
that
\begin{equation}
\label{A2.2}
Q(\xi) = 
\frac{1}{\pi} 
\frac{d}{d\xi} K \left (e^{\pi \xi} \right )  
= \frac{\zeta (\xi)}{\pi} K^\prime (\zeta (\xi)) 
\end{equation}
is a doubly periodic function with periods
periods $2 {\rm i}$ and $2 T$, where $T=(1/\pi) \log \rho^{-1} $, with double pole
at $\xi=0$ and all points congruent to it, i.e., at the set of points
$2 {\rm i} m + 2 n T$ for $(m,n) \in
\mathbb{Z}^2$. Further, from (\ref{1.10}) it follows that 
\begin{equation}
\label{A2.3}
\lim_{\xi \rightarrow 0} \xi^2 Q(\xi) = \frac{1}{\pi^2} 
\lim_{\zeta \rightarrow 1} \left ( \log \zeta \right )^2 \zeta K^\prime (\zeta)  
= -\frac{1}{\pi^2}.
\end{equation}
On applying Liouville's theorem we deduce
\begin{equation}
\label{A2.4}
Q(\xi) = C_0-\frac{1}{\pi^2} \wp \left (\xi; T, {\rm i} \right )  \ ,
\end{equation}
for some constant $C_0$ that will be determined shortly, and where
we have introduced the Weierstrass $\wp$ function represented by
\begin{equation}
\label{A2.5}
\wp (\xi; \omega_1, \omega_2 ) = 
\frac{1}{\xi^2} + \sum_{(m,n) \in \mathbb{Z}^2 \setminus \{(0,0)\}}
\left [\frac{1}{(\xi-2 n \omega_1 - 2m \omega_2 )^2} - \frac{1}{(2 n \omega_1 + 2 m \omega_2)^2} 
\right ].
\end{equation}
This representation is absolutely convergent for any $\xi$ different from $0$ and points congruent to it,
and so the order of summation in $m$ and $n$ is irrelevant.
To sum first in $m$, it is convenient to introduce  
\begin{equation}
\label{A2.6}
\xi_n = \xi-2 n T ~~\ , ~~\xi_{0,n} = -2 n T.
\end{equation}
It follows that
\begin{multline}
\label{A2.7}
\wp (\xi; T, {\rm i}) = 
\sum_{n \in \mathbb{Z}} \sum_{m=1}^{\infty} 
\left [ \frac{1}{(\xi_n - 2 {\rm i} m)^2}
- \frac{1}{(\xi_{0,n} - 2 {\rm i} m )^2} 
+\frac{1}{(\xi_n + 2 {\rm i} m)^2}
- \frac{1}{(\xi_{0,n} + 2 {\rm i} m )^2} 
\right ]  \\
+\frac{1}{\xi^2} + 
\sum_{n \ne 0} \left ( \frac{1}{\xi_n^2} - \frac{1}{4 n^2 T^2} \right )  
\end{multline}
On use of the meromorphic representation of $1/{\sinh^2({\pi \xi/2})}$, it follows that
\begin{equation}
\label{A2.8}
\wp (\xi; T, {\rm i}) 
= \frac{\pi^2}{4 \sinh^2 \frac{\pi \xi}{2} } + \frac{\pi^2}{12} 
+\sum_{n \ne 0} \left ( \frac{\pi^2}{4 \sinh^2 \frac{\pi \xi_n }{2}} - 
\frac{\pi^2}{4 \sinh^2 (\pi n T)} \right ) \ ,   
\end{equation}
which implies  
\begin{equation}
\label{A2.9}
Q (\xi) = -\frac{1}{4 \sinh^2 \frac{\pi \xi}{2} } - \frac{1}{12} 
-\sum_{n \ne 0} \left ( \frac{1}{4 \sinh^2 \frac{\pi \xi_n }{2}} - 
\frac{1}{4 \sinh^2 (\pi n T)} \right ) + C_0 \ ,   
\end{equation}
where $C_0$ will shortly be determined.
From (\ref{1.10}) and (\ref{A2.2}), it follows that
$K (e^{\pi \xi})$ from $Q(\xi)$ may be related through the following integration: 
\begin{equation}
\begin{split}
\label{A2.10.1}
K \left (e^{\pi \xi} \right ) = 
\frac{1}{2} \coth \frac{\pi \xi}{2} 
&+ \sum_{n \ne 0} \left \{ 
\frac{1}{2} \coth \left ( \frac{\pi}{2} (\xi-2 nT) \right ) 
+ \frac{\pi \xi}{4 \sinh^2 (n \pi T)} \right \}
\\ &+ \left ( C_0-\frac{1}{12} \right ) \pi \xi + c.
\end{split}
\end{equation}
Using 
\begin{equation}
\coth z = \frac{e^{2z}+1}{e^{2z}-1} = \frac{2 e^{2z}}{e^{2z} - 1} - 1 
\end{equation} 
it follows that an alternative representation is
\begin{multline}
\label{A2.11}
K\left (e^{\pi \xi} \right ) = 
\frac{e^{\pi \xi}}{e^{\pi \xi} -1 } 
+ \sum_{n=1}^\infty \left \{ 
\frac{e^{\pi \xi} e^{-2 n T\pi}}{ 
e^{\pi \xi} e^{-2 n T\pi}-1}
+\frac{e^{-\pi \xi} e^{-2 n T\pi}}{ 
1-e^{-\pi \xi} e^{-2 n T\pi}} \right \} \\
+\left (c-\frac{1}{2} \right ) +  
\left ( \sum_{n \ne 0} \frac{1}{4 \sinh^2 (n \pi T)} + C_0 - \frac{1}{12} \right ) \pi \xi
\end{multline}
This reproduces the representation (\ref{1.10}) provided, 
\begin{equation}
\label{A2.12}
c=\frac{1}{2} ~~\ ,~~ C_0 = \frac{1}{12} 
- \sum_{n\ne 0} \frac{1}{4 \sinh^2 (n \pi T)} 
\end{equation}
Thus, with the determination of $C_0$, the connection between $Q(\xi)$ 
and 
the Weierstrass $\wp$ function has been made, and with the integration constant
$c$ as determined, it follows that
\begin{equation}
\label{A2.10}
\frac{1}{\pi} K \left (e^{\pi \xi} \right ) = 
+ \frac{1}{\pi^2 \xi} + \frac{1}{2} +
\int_0^{\xi} \left ( Q (\xi') + \frac{1}{\pi^2 {\xi'}^2} \right ) d\xi \ ,
\end{equation}
and hence $K$ can be determined in terms of the Weierstrass $\wp$ function 
through an integration.
To find the uniform asymptotics for small $T$, it is better to sum over $n$ first in the
meromorphic representation (\ref{A2.7}) for $\wp$. 
Then, using the meromorphic representation of ${\rm cosec}^2$, it follows that
\begin{equation}
\label{A2.14}
\wp (\xi; T, {\rm i}) 
= \frac{\pi^2}{4T^2 \sin^{2} \frac{\pi \xi}{2 T}}  - \frac{\pi^2}{12 T^2} 
+ \sum_{m \ne 0} \left \{  
\frac{\pi^2}{4T^2 \sin^{2} \frac{\pi (\xi-2{\rm i} m)}{2 T}} 
+ \frac{\pi^2}{4T^2 \sinh^2 \frac{\pi m}{T} } \right \}
\end{equation}
On using the relation between $K$, $Q$ and $\wp$, we find
\begin{multline}
\label{2.15}
K \left ( e^{\pi \xi} \right ) 
= \frac{1}{2T} \cot \frac{\pi \xi}{2T}  
+ \frac{1}{2T} 
\sum_{m \ne 0} \left \{ 
\cot \left ( \frac{\pi (\xi-2{\rm i}m) }{2T} \right )  - \cot \left ( \frac{-{\rm i} m \pi}{T} \right )
\right \} \\
+ \frac{1}{2} + \left ( \frac{1}{2 \pi T} + C_1 \right ) \pi \xi  \ ,
\end{multline}
where
\begin{equation}
\label{A2.16}
C_1 = \frac{1}{12} - \frac{1}{2} \sum_{n=1}^\infty
\frac{1}{\sinh^2 (n \pi T)} + \frac{1}{12 T^2} 
- \frac{1}{2 T^2} 
\sum_{m=1}^{\infty} \frac{1}{\sinh^2 \frac{m \pi}{T} } 
- \frac{1}{2 \pi T}
\end{equation}
$C_1=0$ as it must in order that $K (e^{\pi (\xi+2 T)} ) = 1 + K (e^{\pi \xi} )$.
This follows indirectly since \eqref{1.10} implies that 
property and representation \eqref{2.15} has been derived in a series of steps from (\ref{1.10}).

It is interesting to note that
independently we may arrive at the same conclusion by
$C_1$ as meromorphic
function of $T$ where the residues vanish at all possible poles and the asymptotics as $T \rightarrow 0$
gives zero, implying from Liouville's theorem that $C_1=0$. 
The single valuedness of $K$ as a function of $\zeta$ is obvious in (\ref{1.10}).
To directly check this property, {\it i.e.} that 
$K \left ( e^{\pi (\xi+2{\rm i})} \right ) = K \left ( e^{\pi \xi} \right )$, is
not so obvious in (\ref{2.15}) since
there is a linear term $\xi = (1/\pi) \log \zeta$. Nonetheless, this is true,
as we now argue. It is useful to introduce notation
\begin{equation}
\label{A2.17}
\chi = e^{i \pi \xi/T}  ~~\ , ~~~ \mu = e^{-2 \pi/T} 
\end{equation} 
Then using the exponential representation of $\cot$ function, \eqref{2.15} implies that
\begin{equation}
\label{A2.18}
K \left ( e^{\pi \xi} \right ) = \frac{\xi}{2T} + \left ( \frac{1}{2} - \frac{{\rm i}}{2 T} \right ) 
+ \frac{{\rm i}}{T} \left ( \frac{\chi}{\chi-1} \right )   
+\frac{{\rm i}}{T} \sum_{m=1}^\infty \left ( 
\frac{\mu^m \chi^{-1} }{1-\mu^m \chi^{-1} } 
-\frac{\mu^m \chi }{1-\mu^m \chi}  
\right ).
\end{equation}
We notice that the transformation $\xi \mapsto \xi+2 {\rm i}$ implies that 
$\chi \mapsto \mu \chi$ and it is then readily checked by shifting the summation index that
$K \left ( e^{\pi (\xi+2 {\rm i})} \right ) = K \left (e^{\pi \xi} \right )$, as it must to be
consistent with the single valuedness of $K(\zeta)$ in the representation (\ref{1.10}).

\end{document}